\documentclass[11pt,a4paper]{article}

\usepackage{amsmath}
\usepackage{amsthm}
\usepackage{amssymb}
\usepackage{mathtools}
\usepackage{fullpage}

\usepackage{graphicx}
\usepackage[colorlinks=true, pdfstartview=FitV, linkcolor=blue, citecolor=blue, urlcolor=blue]{hyperref}

\usepackage[capitalise]{cleveref}

\newtheorem{thm}{Theorem}
\newtheorem{proposition}[thm]{Proposition}
\newtheorem{lemma}[thm]{Lemma}
\newtheorem{cor}[thm]{Corollary}

\theoremstyle{remark}

\newcommand{\conj}[1]{\overline{#1}} 
\newcommand{\ip}[2]{\langle #1,#2\rangle} 
\newcommand{\floor}[1]{\left\lfloor #1 \right\rfloor}
\newcommand{\pa}{\partial}
\newcommand{\eps}{\epsilon}

\newcommand{\T}{\mathbb{T}}
\newcommand{\N}{\mathbb{N}}
\newcommand{\R}{\mathbb{R}}

\newcommand{\OO}{\mathcal{O}}

\newcommand{\ee}{\mathrm{e}} 
\newcommand{\ii}{\mathrm{i}} 
\newcommand{\dd}{\mathrm{d}} 
\newcommand{\init}{\mathrm{in}} 
\newcommand{\jp}{(j{+}1)} 
\newcommand{\lp}{(l{+}1)} 

\newcommand{\tnorm}[1]{{\left\vert\kern-0.25ex\left\vert\kern-0.25ex\left\vert #1
    \right\vert\kern-0.25ex\right\vert\kern-0.25ex\right\vert}}
\newcounter{savecntrP7}
\newcounter{restorecntrP7}

\title{Well-posedness of the Prandtl equation \\ without any structural assumption}
\author{Helge Dietert\setcounter{savecntrP7}{\value{footnote}}\thanks{Universit\'e Paris Diderot, Sorbonne Paris Cit\'e, Institut de Math\'ematiques de Jussieu-Paris Rive Gauche (UMR 7586), F-75205 Paris, France}
	\and David G\'erard-Varet\setcounter{restorecntrP7}{\value{footnote}}%
	\setcounter{footnote}{\value{savecntrP7}}\footnotemark
	\setcounter{footnote}{\value{restorecntrP7}}\; \thanks{Institut Universitaire de France, F-75205 Paris, France}
}
\begin{document}
\maketitle

\begin{abstract}
  We show the local in time well-posedness of the Prandtl equation for data
  with Gevrey $2$ regularity in $x$ and $H^1$ regularity in $y$. The
  main novelty of our result is that we do not make any assumption on
  the structure of the initial data: no monotonicity or hypothesis on
  the critical points. Moreover, our general result is optimal in
  terms of regularity, in view of the ill-posedness result of
  \cite{gerard-varet-dormy-2009-ill-posedness}.
\end{abstract}

\section{Introduction}

We are interested in the 2D Prandtl equation
\begin{equation}
  \label{eq:prandtl}
  \partial_t U^P + U^P \partial_x U^P + V \partial_y U^P - \partial_y^2 U^P = \partial_t U^E +U^E \partial_x U^E,
  \quad \partial_x U^P + \partial_y V^P = 0,
\end{equation}
set in the domain $\Omega = \T \times \R_+$, completed with boundary conditions
\begin{equation}
  \label{eq:BC}
  U^P\vert_{y=0} = V^P\vert_{y=0} = 0, \quad \lim_{y \rightarrow +\infty} U^P = U^E.
\end{equation}
This equation is a degenerate Navier-Stokes model, introduced by
Prandtl in 1904 to describe the boundary layer, which is the region of
high velocity gradients that forms near solid boundaries in
incompressible flows at high Reynolds number. It can be derived from
the Navier-Stokes equation under the formal asymptotics
\begin{equation} \label{eq:Prandtl:asym}
  (u^\nu, v^\nu)(t,x,z) \approx (U^P(t,x,z/\sqrt{\nu}), \sqrt{\nu} V^P(t,x,z/\sqrt{\nu})), \quad (U^P,V^P) = (U^P,V^P)(t,x,y),
\end{equation}
where $\nu$ is the inverse Reynolds number, and $(u^\nu, v^\nu)$ is
the Navier-Stokes solution. This asymptotics is supposed to apply to
the flow in the boundary layer region: the typical scale $\sqrt{\nu}$
of the boundary layer in this model is inspired by the heat part of
the Navier-Stokes equation. Away from the boundary, one rather expects
an inviscid asymptotics of the type
$$ (u^\nu, v^\nu)(t,x,z)  = (u^E, v^E)(t,x,z), $$
where $(u^E, v^E)$ is the solution of the Euler equation. In order to
match the two asymptotic expansions, one must impose the condition
$$ \lim_{y \rightarrow +\infty} U^P(t,x,y) = U^E(t,x) := u^E(t,x,0),  $$
which yields the boundary condition for $y \to \infty$ in
\eqref{eq:BC}. The other two boundary conditions at $y = 0$ express
the usual no-slip condition at the boundary.  We refer to \cite{weinan} for
a more detailed derivation.  Let us stress that the pressure in the
Prandtl model is independent of $y$: its value is given by the
pressure in the Euler flow at $z=0$. This explains the right-hand side
of \eqref{eq:prandtl}, which depends only on $t,x$, and is coherent
with the third boundary condition in \eqref{eq:BC}.

The Prandtl system \eqref{eq:prandtl}-\eqref{eq:BC} is very classical,
as it appears in most textbooks on fluid dynamics. Still, it is
well-known from physicists that its range of applications is narrow,
due to underlying instabilities. Among those instabilities, one can
mention the phenomenon of separation, which is related to the
development of a reverse flow in the boundary layer \cite{gargano1,DalMas}. Another
example is the so-called Tollmien-Schlichting wave, that is typical of
viscous flows at high but finite Reynolds number \cite{DrazinReid,GGN2014}. Of course,
such instability mechanisms create difficulties at the PDE level,
making the mathematical analysis of boundary layer theory an
interesting topic. The two main problems that one needs to address are
the well-posedness of the reduced model \eqref{eq:prandtl}, and the
validity of the asymptotics \eqref{eq:Prandtl:asym}. We shall focus on
the former in the present paper. About the validity of boundary layer
expansions in the unsteady setting, there are many possible references, among which \cite{SaCa2,Gre,WangZhang,GeMaMa,MaeMaz,GNg}. About the steady setting, see the recent works \cite{GuoNg,GerMae,GuoIye}.

To analyse the well-posedness of the Prandtl model is uneasy, even at
the level of local in time smooth solutions. The key difference with
Navier-Stokes is that there is no time evolution for the vertical
velocity, which is recovered only through the divergence-free
condition. Hence, the term $v \partial_y u$ can be seen as a first
order nonlinear operator in $x$. Moreover, this operator is not
skew-symmetric in $H^s$. As the diffusion in \eqref{eq:prandtl} is
only transverse, this prevents the derivation of standard Sobolev
estimates.  The first rigorous study of the Prandtl equation goes back
to Oleinik \cite{Ole}, who tackled the case of data $U\vert_{t=0}$ that
are monotonic in $y$. She established local well-posedness of the
system using the so-called Crocco transform, a tricky change of
variables and unknowns. Let us stress that such monotonicity
assumption excludes the phenomenon of reverse flow and therefore
prevents boundary layer separation. More recently, the local
well-posedness result of Oleinik was revisited using the standard
Eulerian form of the equation, see \cite{AlWaXuTa,masmoudi-wong-2015-local-in-time-prandtl,KuMaViWo} for the local theory in
Sobolev spaces.

The analysis of non-monotonic data is much more recent, and has
experienced some strong impetus over the last years. Surprisingly, it
was shown in \cite{gerard-varet-dormy-2009-ill-posedness} that the Prandtl system is ill-posed in the
Sobolev setting ({\it cf.} \cite{GerNg,GuoNg2,LiuYang} for improvements). Specifically, paper \cite{gerard-varet-dormy-2009-ill-posedness} centers on the linearization of \eqref{eq:prandtl}-\eqref{eq:BC} around shear flows,
given by $(U,V) = (U_s(y),0)$. The linearized system reads
\begin{equation} \label{eq:lin:prandtl}
  \begin{aligned}
    \partial_t u + U_s \partial_x u + U'_s v - \partial^2_y u = 0, \\
    \partial_x u + \partial_y v = 0, \\
    u\vert_{y=0} = v\vert_{y=0} = 0, \quad \lim_{y \rightarrow +\infty} u = 0.
  \end{aligned}
\end{equation}
In the case where $U_s$ has one non-degenerate critical point, one can
show that \eqref{eq:lin:prandtl} has unstable solutions of the form
$u(t,x,y) = \ee^{\ii k x} \ee^{\sigma_k t} u_k(y)$ for $k$ arbitrarily
large and $\Re \sigma_k \sim \lambda \sqrt{k}$. Such high frequency
instability forbids the construction of Sobolev solutions.  To obtain
positive results, one must start from initial data $u_\init$ that are
strongly localized in Fourier, typically for which
$|\hat{u}_0(k,y)| \lesssim \ee^{-\delta |k|^\gamma}$ for some positive
$\delta > 0$, $\gamma \le 1$. Such localization condition corresponds
to Gevrey regularity in $x$ (Gevrey class $1/\gamma$).  The first
result in this direction is due to Sammartino and Caflisch
\cite{sammartino-calfisch-1998-boundary-layer-analytic}, who
established existence of local in time solutions in the analytic
setting ($\gamma = 1$). See also the nice paper \cite{KuVi}. Note that the
requirement for analyticity is natural in view of standard
estimates. For instance, at the level of the linearized equation
\eqref{eq:lin:prandtl}, one gets directly by testing against $u$ that
$$ \partial_t \|\hat{u}(t,k,\cdot)\|_{L^2_y} \le  C |k|  \, \|\hat{u}(t,k,\cdot)\|_{L^2_y}  $$
so that
$\|\hat{u}(t,k,\cdot)\|_{L^2_y} \le \ee^{C|k|t} \|\hat{u}_0(k,
\cdot)\|_{L^2_y}$.  Hence, if
$\|\hat{u} _0(k,\cdot)\|_{L^2_y} \lesssim \ee^{-\delta |k| }$, a
uniform control will be provided as long as $t \le \delta/C$.

To relax the analyticity condition is much harder. In the special case
where $u_\init$ \emph{has for each value of $x$ a single non-degenerate
  critical point in $y$}, the first author and N. Masmoudi proved the
local well-posedness of system \eqref{eq:prandtl}-\eqref{eq:BC} for
data that are in Gevrey class $7/4$ with respect to
$x$ \cite{GeMa}. Well-posedness was extended to Gevrey class $2$ in article
\cite{LiYang}, for data that are small perturbations of a shear flow with a
single non-degenerate critical point. Note that this exponent
(corresponding to $\gamma = 1/2$) is optimal in view of the
instability mechanism of \cite{gerard-varet-dormy-2009-ill-posedness}.

All the recent results mentioned above rely heavily on the structure
of the initial data: monotonicity for the Sobolev setting, single
non-degenerate critical points for the Gevrey setting. It is therefore
natural to ask about the optimal regularity under which local
well-posedness of the Prandtl equation holds, without additional
structural assumption. This is the problem that we solve in the
present paper: we establish the short-time well-posedness of the
Prandtl equation for general data with Gevrey $2$ regularity in $x$
and $H^1$ regularity in $y$. We recall once more that such regularity
framework is the best possible. Indeed, from the results of \cite{gerard-varet-dormy-2009-ill-posedness},
high frequency modes $k$ in $x$ may experience exponential growth with
rate $\sqrt{k}$. This means that to hope for short time stability, the
amplitude of these modes should be $O(\ee^{-C\sqrt{k}})$, which is the
Fourier translation of a Gevrey 2 requirement.

\section{Result}

Let $\gamma \ge 1$, $\tau > 0$, $r \in \R$. For functions $f = f(x)$
of one variable, we define the Gevrey norm
\begin{equation}\label{eq:Gevrey:1d}
  |f|_{\gamma,\tau, r}^2 =  \sum_{j \in \N}
  \left( \frac{\tau^{j+1} \jp^r}{(j!)^{\gamma}} \right)^2
  \| f^{(j)} \|^2_{L^2(\T)}
\end{equation}
and for  functions $f = f(x,y)$ of two variables, the norm
\begin{equation} \label{eq:Gevrey:2d}
  \| f \|_{\gamma,\tau, r}^2 = \sum_{j \in \N}
  \left( \frac{\tau^{j+1} \jp^r}{(j!)^{\gamma}} \right)^2
  \| \partial_x^j  f \|^2_j,
\end{equation}
where $\|\cdot\|_j$, $j \ge 0$, denotes a family of weighted $L^2$
norms. Namely,
\begin{equation} \label{norm:rho}
  \| f \|_j^2 =   \int_{\T\times\R^+} |f(x,y)|^2 \rho_j(y)\, \dd x\, \dd y,
\end{equation}
where $\rho_j$, $j \ge 0$, is the family of weights given by
\begin{equation*}
  \rho_0(y) = (1+y)^{2m}, \quad
  \rho_j(y) = \frac{\rho_{j-1}(y)}{\left(1+\frac{y}{j^\alpha}\right)^2}
    = \rho_0(y) \prod_{k=1}^j \left(1 + \frac{y}{k^{\alpha}}\right)^{-2}, \quad j  \ge 1,
\end{equation*}
for fixed constants $\alpha \ge 0$ and $m \ge 0$ chosen later ($m$
large enough and $\alpha$ matching the constraints found from the
estimates). The need for this family of weights will be clarified later. Let us note that {\em locally in $y$}, this family of norms is comparable to more classical families such as
\begin{equation}
  ||| f |||_{\gamma,\tau, r}^2 = \sum_{j \in \N}  \left( \frac{\tau^{j+1} \jp^r}{(j!)^{\gamma}} \right)^2  \| \partial_x^j  f \|^2_{L^2}.
\end{equation}
For instance, for functions $f$ which are zero for $|y|\ge M$, one has
$$  \| f \|_{\gamma, \tau,r} \le C_M ||| f |||_{\gamma,\tau,r}, \quad ||| f |||_{\gamma,\tau,r} \le C_{M,\tau'} \| f \|_{\gamma,\tau',r} \text{ for any $\tau' > \tau$}. $$
The only difference is when $y$ goes to infinity, where the family of weights $\rho_j$ puts less constraints on the decay of the derivatives  compared to a fixed weight $\rho_0$ for derivatives of any order.

With these spaces, we can now state our main result.
\begin{thm} \label{thm:main}
  There exists $m$ and $\alpha$ such that: for all
  $0 < \tau_1 < \tau_0$, $r \in \R$, for all $T_0 > 0$, for all $U^E$ satisfying
  \begin{equation*}
    \sup_{[0,T_0]} \,  |\pa_t U^E|_{2,\tau_0,r} + |U^E|_{2,\tau_0,r} < +\infty,\quad
    \sup_{[0,T_0]}   \max_{l=0,\dots,3} \| \partial_t^l(\partial_t + U^E\partial_x) U^E
    \|_{H^{6-2l}(\T)} < +\infty
  \end{equation*}
  for all  $U^P_\init$ satisfying
  $$ \| U^P_\init-U^E\vert_{t=0} \|_{2,\tau_0,r} < +\infty,
  \quad \|(1+y)\pa_y U^P_\init \|_{2,\tau_0,r} < +\infty,
  \quad \|(1+y)^{m+6}\pa_y U^P_\init \|_{H^6(\T \times \R_+)} < +\infty $$
  and under usual compatibility conditions (see the last remark below), there exists
  $0 < T \le T_0$ and a unique solution $U^P$ of
  \eqref{eq:prandtl}-\eqref{eq:BC} over $(0,T)$ with initial data $U^P_\init$ that
  satisfies
  \begin{equation*}
    \sup_{t \in [0,T]} \|U^P(t)-U^E(t) \|^2_{2,\tau_1,r}
    +  \sup_{t \in [0,T]} \|(1+y)\pa_y U^P(t) \|^2_{2,\tau_1,r}
    + \int_0^T  \|(1+y)\pa^2_y U^P(t) \|^2_{2,\tau_1,r} \, \dd t < +\infty
  \end{equation*}
\end{thm}

\textbf{Remarks.}
\begin{itemize}
\item The main novelty of the theorem is that we reach the optimal
  Gevrey regularity although no structural assumption is made on the
  data: no monotonicity, or hypothesis on the number and order of the
  critical points is needed. Only Gevrey regularity of the data and
  natural compatiblity conditions are required.
\item Our method of proof, explained below, is inspired by the
  hyperbolic part of the Prandtl equation. It is based on both a
  tricky change of unknown and appropriate choice of test
  function. This method would also allow to recover the Sobolev
  well-posedness of the hyperbolic version of the Prandtl system {\em
    by means of energy methods}. As far as we know, the well-posedness
  of this inviscid Prandtl equation had been only established in $C^k$
  spaces using the method of characteristics: see \cite{HH03} for
  more. This part will be detailed elsewhere. In the case of the usual
  Prandtl equation studied here, our methodology has to be slightly
  modified to handle in an optimal way the diffusion term. Still,
  commutators are responsible for the loss of Sobolev regularity: only
  Gevrey $2$ smoothness in $x$ can be established.
  \item There is a loss on the Gevrey radius $\tau$ of the solutions
  through time, going from $\tau_0$ to $\tau_1$. This loss, which
  appears technical in the paper, is actually unavoidable. This is due
  to the instabilities described in \cite{gerard-varet-dormy-2009-ill-posedness}: exponential growth of
  perturbations at rate $\sqrt{k}$ causes a decay of the Gevrey radius
  linearly with time.
\item Besides the regularity requirements mentioned in Theorem
  \ref{thm:main}, the initial data must satisfy {\em compatibility
    conditions}. It is typical of parabolic problems in domains with
  boundaries, {\it cf.} \cite[Chapter 3]{Metivier} for a general
  discussion. Here, the value of $U^P_\init$ and of some of its
  derivatives at $y=0$ cannot be arbitrary: they must be related to
  $U^E$ accordingly to the equation and to the amount of regularity
  asked for $u$ (with respect to the $y$-variable). Let us note that
  locally near $y=0$, most of our estimates only involve $U^P-U^E$ in
  $L^2_t H^2_y$ (not mentioning the Gevrey regularity in $x$). Such
  estimates could be carried with the single compatibility condition
  $U^P_\init\vert_{y=0}=0$. Still, the low norm $\| (U^P, V^P) \|_{low}$
  introduced in \eqref{def:low} involves more $y$-derivatives: its
  control through \cref{lemma:low} implies therefore a few more
  compatibility conditions. For the sake of brevity, we do not provide
  their explicit expressions, and refer to \cite[Proposition
  2.3]{XuZhang} for a detailed discussion on a variation of the
  Prandtl equation.

\end{itemize}

\textbf{Outline of the strategy.}
As mentioned earlier, our analysis of the Prandtl equation relies on
the identification of new controlled quantities because the usual
unknown $u$ and kinetic energy do not give enough information.  To
help to identify the relevant quantities, it is a good idea to start
from the study of the linearized system \eqref{eq:lin:prandtl}. After
Fourier transform in $x$ and Laplace transform in time, we are left
with the ODE
\begin{equation} \label{OS}
  (\lambda + \ii k U_s) \partial_y \Psi - \ii k U'_s  \Psi  - \partial^3_y \Psi =  u_\init
\end{equation}
where $\Psi$ corresponds to the Fourier-Laplace transform of the
stream function. At high frequencies $k$, a natural idea (although not
legitimate in the end) is to neglect the diffusion term. We are then
left with the first order ODE
\begin{equation}  \label{Ray}
  (\lambda + \ii k U_s) \partial_y \Psi - \ii k U'_s  \Psi  =  u_\init.
\end{equation}
We note that the standard estimate (based on taking
$\partial_y \overline{\Psi}$ as a test function) yields a control of
the type
\begin{align*}
  \Re\lambda\, \|\partial_y \Psi \|_{L^2}^2
  &\lesssim k \|\partial_y \Psi \|_{L^2} \|U'_s \Psi \|_{L^2} + \|u_\init \|_{L^2} \|\partial_y \Psi \|_{L^2} \\
  &\lesssim k  \|\partial_y \Psi \|_{L^2}^2 + \|u_\init \|_{L^2} \|\partial_y \Psi \|_{L^2}
\end{align*}
where the last line comes from the Hardy inequality (as soon as
$|U'_s(y)| = \OO(y^{-1})$ at infinity). Such bound ensures the
solvability of the resolvent equation \eqref{Ray} only for
$\lambda \sim k$. This in turn yields a semigroup bound of the type
$\ee^{C kt}$, only compatible with stability in the analytic setting.

To reach stability in lower regularity, an important point is to
notice that the homogeneous equation has
$\Psi_s = (\lambda + i k U_s)$ as a special solution. With the
integrating factor method in mind, it is then natural to set
$\Psi = (\lambda + i k U_s) \psi$. The first order equation
\eqref{Ray} becomes
$$ (\lambda + i k U_s)^2 \partial_y \psi =  u_\init $$
which is much better than the original formulation. Indeed, we can
test the equation against
$\phi = \frac{1}{\lambda + i k U_s} \partial_y \conj{\psi}$ to
obtain a control of $\partial_y \psi$ in terms of $u_\init$, and from
there a control of $\Psi$ for any $\lambda > 0$.

Back to the full resolvent equation \eqref{OS} we find for the same unknown $\psi$
\begin{equation*}
  (\lambda + \ii k U_s)^2 \partial_y \psi
  - (\lambda + \ii k U_s) \partial_y^3 \psi
  = u_\init + [\lambda+\ii k U_s,\partial_y^3] \psi.
\end{equation*}
Testing again against
$\phi = \frac{1}{\lambda + i k U_s} \partial_y \overline{\psi}$, the
LHS allows the control
\begin{equation*}
  \Re \lambda\, \| \partial_y \psi \|_{L^2}^2
  + \| \partial_y^2 \psi \|_{L^2}^2.
\end{equation*}
In the commutator at the RHS,  the worst error term is
$3\ii k \partial_y U_s \partial_y^2 \psi$, which is bounded as
\begin{equation*}
  C \frac{k}{|\Re \lambda|}
  \| \partial_y \psi \|_{L^2}
  \| \partial_y^2 \psi \|_{L^2}
  \le \frac 12 \| \partial_y^2 \psi \|_{L^2}^2
  + \frac{C^2k^2}{|\Re \lambda|^2}
  \| \partial_y \psi \|_{L^2}^2.
\end{equation*}
We see that under the constraint $\Re \lambda \sim k^{2/3}$, the
estimate can be closed, and this can be shown to imply short time
stability for data with Gevrey regularity $3/2$. This estimate around
a shear flow is detailed as Lemma 4.1 in \cite{dalibard-dietert-gerard-marbach-2018-ibl}.

In order to reach the optimal Gevrey exponent $2$, we need to get rid
of the commutator term containing $\partial_y^2 \psi$, which comes
with a worse control than $\partial_y \psi$. To do so, we change a bit
our new unknown $\psi$: we now define $\psi$ through the relation
\begin{equation} \label{def:auxiliary1}
  \Psi = (\lambda+\ii k U_s - \partial_y^2) \psi
\end{equation}
including the diffusion term. Hence,  \eqref{OS} becomes
\begin{equation*}
  (\lambda + \ii k U_s - \partial_y^2)^2 \partial_y \psi
  + (\lambda + \ii k U_s - \partial_y^2)
  (\ii k U_s' \psi)
  - \ii k U_s
  (\lambda + \ii k U_s - \partial_y^2) \psi = u_\init.
\end{equation*}
Testing this time against the solution $\phi$ of
$(\lambda+\ii k U_s - \partial_y^2) \phi = \partial_y \conj{\psi}$
(again with the diffusion term), the LHS yields the same control, but
the error term is now
\begin{equation*}
  \Re \int [\lambda+\ii k U_s - \partial_y^2, \ii k U_s']
  \psi\; \phi.
\end{equation*}
From the definition of $\phi$ it can be shown that
$\| \phi \| \lesssim \lambda^{-1} \| \partial_y \psi \|$ so that the error can be bounded by
\begin{equation*}
  \frac{k}{|\Re\lambda|} \| \partial_y \psi \|_{L^2}^2.
\end{equation*}
The estimate can now be closed for $\Re\lambda \sim k^{1/2}$ yielding
Gevrey regularity $2$.

Obviously, such approach is no longer applicable as such to the
nonlinear system \eqref{eq:prandtl}-\eqref{eq:BC}: we not only lose
the linearity of the equations, but the coefficients are no longer of
shear flow type. They notably depend on $t$ and $x$, which forbids an
easy use of Fourier or Laplace transform.  Rather than turning to the
characterization of Gevrey spaces in the Fourier variable $k$, we
consider norms based on the $x$-variable, see \eqref{eq:Gevrey:1d} and
\eqref{eq:Gevrey:2d}.  Roughly, the idea is to work with time
dependent norms, that is with the quantities
$$ \|(U^P - U^E)(t) \|_{\gamma,\tau(t),r}, \quad \tau(t) = \tau_0 \ee^{-\beta t}. $$
By differentiating $j$-times the Prandtl equation, we can derive an
equation on
\begin{equation*}
  u_j(t) \: := \:
  \frac{\tau(t)^{j+1} \jp^r}{(j!)^{\gamma}}
  \partial_x^j \left( U^P(t) - U^E(t) \right)
\end{equation*}
that can be written as
\begin{equation} \label{eq:linear:uj}
  (\partial_t + \beta (j+1) ) u_j + U^P \partial_x u_j
  + V^P \partial_y u_j  + v_j \partial_y U^P -\partial_y^2 u_j = F_j,
  \quad v_j = - \int_0^y \partial_x u_j.
\end{equation}

Roughly, inspired by the shear flow case, the idea will be to introduce
as a new unknown  the solution $\psi_j = \int_0^y H_j$ of
\begin{equation*}
  (\partial_t + \beta \jp + U^P \partial_x - \partial_y^2)
  \psi_j = \int_0^y u_j\, \dd z.
\end{equation*}
which is reminiscent of the Fourier relation \eqref{def:auxiliary1}. The test function $\phi_j$ should then solve the reverse equation
\begin{equation*}
  (-\partial_t + \beta \jp - U^P \partial_x - \partial_y^2)
  \phi_j = \partial_y \psi_j
\end{equation*}
and be solved backward in time. Performing the same estimate as in
the shear flow case, we expect to find an inequality of the type
\begin{equation*}
  \beta \jp \| \partial_y \psi_j \|^2
  + \| \partial_y^2 \psi_j \|^2
  \lesssim
  \frac{1}{\beta^3 \jp^3} \| F_j \|^2
  + \frac{1}{\beta^3 \jp^3} \| \partial_x \partial_y \psi_j \|^2
\end{equation*}
By exploiting a relation of the form
$\|\partial_x \partial_y \psi_j\| \sim j^\gamma \| \partial_y
\psi_{j+1}\|$ (that needs to be shown!) and using that
$\gamma \le 2$, we will then be able to sum over $j$ and establish for
large enough $\beta$ a control of $\sum_j \| \partial_y \psi_j\|^2$ in
terms of $\sum_j \| F_j \|^2$.

In fact, in implementing this strategy, several refinements are
necessary, and the relations satisfied by $\psi_j = \int_0^y H_j$ or $\phi_j$ need to be slightly modified. Particularly problematic is the term $V^P \partial_y$
because $V^P \sim -\partial_x U^E y$ increases linearly with
$y$: this prevents from closing an energy estimate with a fixed weigth
$\rho = \rho(y)$. This difficulty appears in various places in the
literature on the Prandtl equation. This is for instance the reason
why article \cite{GeMa} is limited to the special case $U^E = 0$ and
decaying initial data. One can also mention \cite{KuVi}, where this
difficulty is overcome by a clever change of variables, which is
reminiscent of the method of characteristics and allows to remove the
bad part of the convection term from the equation. Energy estimates
can then be established in these new coordinates $x',y'$, and yield
some local well-posedness result, with solutions that are analytic in
$x'$ and $L^2$ in $y'$. The disadvantage of this approach is that the
regularity of the solution in the original variables $x$ and $y$ is no
longer clear at positive times. Here, we stick to the eulerian variables, but overcome
the difficulty by introducing the family of weights $\rho_j$,
$j \ge 0$. These weights allow to trade a power of $y$ against a
derivative in $x$, which is appropriate to the commutator
terms. Moreover, they put very little conditions on the derivatives of
the solution, so that they provide a very general framework for
well-posedness. Note that the specific expression of $\rho_j$ is
important: it could not be for instance replaced by the more natural
guess $(1+y)^{2(m-j)}$, as commutators with the diffusion term would
not be under control. Note also that the strategy used in \cite{masmoudi-wong-2015-local-in-time-prandtl},
where Sobolev well-posedness is established under monotonicity
assumptions by increasing the weight with the number of
$y$-derivatives, does not extend to the Gevrey framework in variable
$x$.

The plan of the paper is as follows. In the next section, we first
collect several properties of the weight $\rho_j$.  We then write the
equations satisfied by the $x$-derivatives of the Prandtl solution in
a form analogue to \eqref{eq:linear:uj}. This means that we put most
of the nonlinear terms at the right-hand side, and consider those
equations as linear. We finish the section by introducing the adapted
quantities $H_j$ and $\phi_j$.  The main section is Section
\ref{sec:linear}: {\it a priori} Gevrey estimates for the linear equations are
perfomed, that provide a control of the $u_j$'s in terms of the
nonlinear terms $F_j's$. Note that such estimates are obtained under a
condition of the form $\beta > C (1+ \|(U^P, V^P) \|_{low})^2$, where
$\|(U^P, V^P)\|_{low}$ is a low regularity norm of the solution.  The
treatment of the nonlinearity $F_j$ is then handled in Section
\ref{sec:nonlinear}.  The last step in the derivation of \textit{a
  priori} estimates is to recover the control of the low regularity
norm $\|(U^P,V^P)\|_{low}$, see \cref{sec:low-norm}. Finally, issues
regarding the construction and uniqueness of solutions are discussed in
Section \ref{sec:exist_unique}.

\section{Preliminaries} \label{sec:prelim}

The explicit form of the weights $\rho_j$ is only needed in the
Section \ref{sec:nonlinear}. In the other parts, we just need a
sufficient control of the logarithmic derivative
(\cref{thm:logarithmic-rho-bound}), a bound for antiderivatives
(\cref{thm:poincare-y-rho}) and relate $\rho_j$ to $\rho_{j+1}$
(\cref{thm:rho-j-j-1}).

\begin{lemma}
  \label{thm:logarithmic-rho-bound}
  Let $m \ge 0$ and $\alpha \ge 0$. There exists a constant
  $C_l$ such that for all $y \in \R^+$, $j \in \N$
  \begin{equation*}
    \left| \frac{\partial_y \rho_j(y)}{\rho_j(y)} \right|
    \le
    \begin{cases}
      C_l \jp^{1-\alpha} & \text{if } \alpha < 1 \\
      C_l \log\jp    & \text{if } \alpha = 1 \\
      C_l    & \text{if } \alpha > 1
    \end{cases}
  \end{equation*}
  and
  \begin{equation*}
    (1+y) \left| \frac{\partial_y \rho_j(y)}{\rho_j(y)} \right|
    \le C_l \, \jp.
  \end{equation*}
\end{lemma}
\begin{proof}
  Given the explicit form of $\rho_j$, we can compute the logarithmic
  derivative of $\rho$ directly as
  \begin{equation*}
    \frac{\partial_y \rho_j}{\rho_j}
    = \partial_y \log \rho_j
    = \partial_y \log \rho_0 - 2 \sum_{k=1}^{j}
    \partial_y \log\left(1 + \frac{y}{k^\alpha}\right)
    = \frac{2m}{1+y} - 2 \sum_{k=1}^{j}
    \frac{1}{k^\alpha \left(1+\frac{y}{k^\alpha}\right)}.
  \end{equation*}
  From this expression the result follows directly.
\end{proof}
\begin{lemma}
  \label{thm:poincare-y-rho}
  For $m > \frac 12$ introduce the constant
  \begin{equation*}
    C_m = \sqrt\frac{1}{2m-1}.
  \end{equation*}
  Then for all $\alpha \ge 0$, $j \in \N$  and all $f = f(y)$,
  \begin{equation*}
    \sup_{y \ge 0}
    \left(\frac{\rho_j(y)}{\rho_0(y)}\right)^{1/2}
    \int_0^y |f(z)|\, \dd z
    \le C_m \, \| f \|_{L^2(\rho_j)}.
  \end{equation*}
 More generally, for $0 \le n \le j$ with $n < m - \frac 12$ one has
  \begin{equation*}
    \sup_{y \ge 0}
    \left(\frac{\rho_j(y)}{\rho_n(y)}\right)^{1/2}
    \int_0^y |f(z)|\, \dd z
    \le C_{m-n} \, \| f \|_{L^2(\rho_j)}.
  \end{equation*}
 Eventually, for all $A = A(x,y)$ and $B = B(x,y)$, the following inequality holds:
  \begin{equation*}
    \| A \int_0^y B(z)\, \dd z \|_j
    \le C_m \| A \|_{L^\infty_xL^2_y(\rho_0)}
    \| B \|_{j}.
  \end{equation*}
\end{lemma}
\begin{proof}
  Note that $\rho_j/\rho_n$ for $j \ge n$ is non-increasing. Hence
  \begin{equation*}
    \left(\frac{\rho_j(y)}{\rho_n(y)}\right)^{1/2}
    \int_0^y |f(z)|\, \dd z
    \le \int_0^y \left(\frac{\rho_j(z)}{\rho_n(z)}\right)^{1/2}
    |f(z)| \, \dd z
    \le \| f \|_{L^2(\rho_j)}
    \left(
      \int_0^y \frac{1}{\rho_n(z)} \, \dd z
    \right)^{1/2},
  \end{equation*}
  where we used the Cauchy-Schwarz inequality in the second inequality.

  As $\alpha \ge 0$ we find directly that
  \begin{equation*}
    \frac{1}{\rho_n(y)}
    \le (1+y)^{-2m}
    \prod_{k=1}^j \left(1 + \frac{y}{k^{\alpha}}\right)^{2}
    \le \frac{1}{(1+y)^{2m-2n}}
  \end{equation*}
  whose integral over $y \in \R^+$ gives $C_{m-n}^2$. This proves the
 first and second bounds.   The remaining estimate with $A$ and $B$ follows directly.
\end{proof}

The weights are decaying so that $\rho_j \le \rho_k$ for $j \ge k$. As
$\alpha \ge 0$, we have for $j \in \N$ that
$(1+y)^2 \rho_{j+1} \le \jp^{2\alpha} \rho_j$ and
$\rho_{j+1} \ge \frac{\rho_j}{(1+y)^2}$. This shows:
\begin{lemma}
  \label{thm:rho-j-j-1}
  Let $\alpha \ge 0$. For $j\in\N$, $A = A(x,y)$ and $B = B(x,y)$ it
  holds that for
  \begin{equation*}
    \| A \|_{j+1} \le \| A \|_j,\qquad
    \| (1+y) A \|_{j+1} \le \jp^{\alpha} \| A \|_j
  \end{equation*}
  and
  \begin{equation*}
    \left\| \frac{A}{(1+y)} \right\|_j \le \| A \|_{j+1},
    \qquad
    \| A \int_0^y B(z)\, \dd z \|_j
    \le C_m \| (1+y) A \|_{L^\infty_xL^2_y(\rho_0)}
    \| B \|_{j+1}.
  \end{equation*}
\end{lemma}

Let us insist again that most parts of the proof would work with
constant weight $\rho$ instead of $\rho_j$. The dependency on $j$ will
be only needed to treat the commutator terms coming from
$V^P \partial_y U^P$. The difficulty is that $V^P$ grows like $y$ as
soon as $U^E$ is non-constant. Here the crucial property that we will
use is that we can control $\| (1+y) A \|_{j+1}$ by
$\jp^{\alpha} \| A \|_j$.

The Prandtl equation is given for $(U^P,V^P)$ with inhomogeneous
boundary conditions at $y \to \infty$. In order to work with
homogeneous boundary conditions at zero and infinity, we introduce
\begin{equation*}
  U^e(t,x,y) = (1-\ee^{-y}) U^E(t,x), \quad
  V^e(t,x,y) = - (y + \ee^{-y}-1)\partial_x U^E(t,x)
\end{equation*}
and set $u = U^P - U^e$, $v = -\int_0^y \partial_x u = V^P - V^e$. Then,
\begin{equation} \label{eq:u}
  \partial_t u + (u \partial_x + v \partial_y) u   + (U^e \partial_x + V^e \partial_y) u
  + (u \partial_x  + v \partial_y) U^e - \partial^2_y u = f^e
\end{equation}
where
\begin{equation} \label{def:fe}
  f^e = \partial_t U^E   + U^E \partial_x U^E - \partial_t U^e - U^e \partial_x U^e
  - V^e \partial_y U^e + \partial^2_y U^e.
\end{equation}
In the new variables $(u,v)$ the boundary conditions are
\begin{equation} \label{eq:BC:bis}
  u = v  = 0 \quad \text{ at } y=0,\qquad
  \text{ and } \lim_{y\to\infty} u =0.
\end{equation}
The condition at $y\to \infty$ will be encoded in the functional space
of $u$.

To prove Theorem \ref{thm:main}, the point is to obtain good estimates
for Gevrey norms of $u$ of type \eqref{eq:Gevrey:2d} for
time-dependent radius $\tau = \tau(t)$. More precisely, we give
ourselves parameters $m,\alpha,\gamma,r$, to be fixed later, as well
as the time-dependent radius $\tau(t) = \tau_0 \ee^{-\beta t}$, with
$\beta > 0$ to be fixed later. Then, for any function $f = f(t,x)$ or
$f=f(t,x,y)$ and $j \in \N$ we set
\begin{equation*}
  f_j(t,\cdot)  :=  M_j\, \partial_x^j  f(t,\cdot)
  \quad \text{with} \quad
  M_j := \frac{\tau(t)^{j+1} \jp^r}{(j!)^{\gamma}}.
\end{equation*}
Taking $j$ derivatives in $x$ of \eqref{eq:u} and multiplying by $M_j$ yields
\begin{equation}
  \label{eq:u-j}
  \Big(\partial_t + \beta \jp + U^P \partial_x + \jp \partial_x U^P
  + V^P \partial_y - \partial_y^2\Big) u_j
  + \partial_y U^P v_j + j \partial_{xy} U^P \partial_x^{-1} v_j
  = F_j
\end{equation}
where $F_j$ collects all terms with less than $j$ derivatives in $x$
as well as the weighted derivative of the forcing $f^e$. It is given by
\begin{equation*}
  \begin{aligned}
    F_j
    = f_j^e &+ M_j \left[u \partial_x, \partial_x^j\right] u
    + M_j \, \partial_x u\, \pa_x^j u \\
    &+ M_j \left[\partial_y u, \partial_x^j\right] v
    + M_j j\, \partial_{xy} u\, \partial_x^{j-1} v
    + M_j v\, \partial_x^j \partial_y u \\
    &+  M_j \left[U^e \partial_x, \partial_x^j\right] u
    + M_j j\, \partial_x U^e\, \partial_x^j u \\
    &+ M_j \left[V^e \partial_y, \partial_x^j\right] u \\
    &+ M_j \left[\partial_x U^e, \partial_x^j\right] u \\
    &+ M_j \left[\partial_y U^e, \partial_x^j \right] v
    + M_j j\, \partial_{xy} U^e\, \partial_x^{j-1} v.
  \end{aligned}
\end{equation*}

We now introduce our crucial auxiliary functions  $H_j(t,x,y)$ defined by
\begin{equation}
  \label{eq:def-h-j}
  \begin{lgathered}
    \Big(\partial_t + \beta \jp + U^P \partial_x + \jp \partial_x U^p
    + V^P \partial_y - \partial_y^2\Big)
    \int_0^y H_j \, \dd z
    = \int_0^y u_j \, \dd z, \\
    H_j|_{t=0} = 0, \qquad
    \partial_y H_j|_{y=0} = 0, \qquad
    H_j|_{y\to \infty} = 0.
  \end{lgathered}
\end{equation}
For the existence of \(H_j\), one can consider \eqref{eq:def-h-j} as a
convection-diffusion equation for \(A_j = \int_0^y H_j\, \dd z\), with
boundary conditions \(A_j|_{y=0} = \partial_y A_j|_{y\to\infty} = 0\), 
which has a solution by the classical theory of parabolic PDEs. The
PDE \eqref{eq:def-h-j} itself then implies that \(\partial_y^2
A_j|_{y=0}\) so that taking \(H_j = \partial_y A_j\) gives the
required solution.

We further introduce the corresponding test functions $\phi_j$ by
\begin{equation}
  \label{eq:def-phi-j}
  \begin{lgathered}
    \left(-\partial_t + \beta \jp - U^P \partial_x + j \partial_x U^p
      - V^P \partial_y - \partial_y V^P - V^P \frac{\partial_y \rho_j}{\rho_j}
      - \left(\partial_y + \frac{\partial_y \rho_j}{\rho_j}\right)^2 \right)
    \phi_j = H_j, \\
    \phi_j|_{t=T} = 0, \qquad
    \phi_j|_{y=0} = 0, \qquad
    \phi_j|_{y\to \infty} = 0.
  \end{lgathered}
\end{equation}
Note here that the operator acting on $\phi_j$ is the formal adjoint
operator of the operator acting on $\int_0^y H_j \, \dd z$ in
\eqref{eq:def-h-j}, {\em with respect to the $L^2(\rho_j)$ scalar product, denoted $\ip{\,}{}_j$}. This is a backward heat equation solved backward
in time for $t \in [0,T]$.

Testing \eqref{eq:def-phi-j} against $\phi_j$ in $\| \cdot \|_j$ and
integrating over $[t,T]$ yields
\begin{equation*}
  \begin{aligned}
    &\frac{1}{2} \| \phi_j(t) \|_j^2
    + \beta \jp \int_t^T \| \phi_j(s) \|_j^2\, \dd s
    + (j{+}\frac 12) \int_t^T
    \ip{\partial_x U^P \phi_j}{\phi_j}_j\, \dd s \\
    &- \frac 12 \int_t^T \ip{\left(\partial_y V^P + V^P
        \frac{\partial_y \rho_j}{\rho_j}\right) \phi_j}{\phi_j}_j \, \dd s
    + \int_t^T \| \partial_y \phi_j(s) \|_j^2 \, \dd s
    + \int_t^T \ip{\frac{\partial_y \rho_j}{\rho_j} \phi_j}{\partial_y
      \phi_j}_j\, \dd s \\
    &= \int_t^T \ip{H_j}{\phi_j}_j\, \dd s.
  \end{aligned}
\end{equation*}
Hence we find
\begin{equation*}
  \begin{aligned}
    &\frac{1}{2} \| \phi_j(t) \|_j^2
    + \frac{3\beta \jp}{4} \int_t^T \| \phi_j(s) \|_j^2\, \dd s
    + \frac{1}{2} \int_t^T \| \partial_y \phi_j(s) \|_j^2 \, \dd s \\
    &\le
    \frac{1}{\beta\jp} \int_t^T \|H_j(s)\|_j^2\, \dd s
    + \left(
      (j{+}\frac 12) \| \partial_x U^P \|_{\infty}
      + \frac 12
      \left\| \partial_y V^P {+} V^P \frac{\partial_y \rho_j}{\rho_j} \right\|_{\infty}
      + \frac 12 \left\| \frac{\partial_y \rho_j}{\rho_j} \right\|_{\infty}^2
    \right) \int_t^T \| \phi_j(s) \|_j^2\, \dd s.
  \end{aligned}
\end{equation*}
By \cref{thm:logarithmic-rho-bound}, under the condition $\alpha \ge \frac 12$,  we get the
following control:
\begin{lemma}
  \label{thm:control-phi-by-h}
  Fix $m\ge 0$ and $\alpha \ge \frac 12$. Then there exist a constant
  $\mathcal{C} = \mathcal{C}(m,\alpha)$ such that for all $j \in \N$
  it holds that
  \begin{equation*}
    \| \phi_j(t) \|_j^2
    + \beta \jp \int_t^T \| \phi_j(s) \|_j^2\, \dd s
    +  \int_t^T \| \partial_y \phi_j(s) \|_j^2 \, \dd s \\
    \le
    \frac{2}{\beta\jp} \int_t^T \|H_j(s)\|_j^2\, \dd s
  \end{equation*}
  if
  \begin{equation*}
    \beta \ge \mathcal{C}
    \left(
      1 + \| \partial_x U^P \|_{\infty}
      + \| \partial_y V^P \|_{\infty}
      + \left\| \frac{V^P}{1+y} \right\|_{\infty}
    \right).
  \end{equation*}
\end{lemma}
Note that for $\alpha < \frac 12$, the term with
$\| \frac{\partial_y \rho_j}{\rho_j} \|_{\infty}^2$ could not have
been absorbed. This \textit{a priori} estimate also ensure the
existence of $\phi_j$ as solution of \eqref{eq:def-phi-j}. A similar
estimate holds for $H_j$ which ensures the existence of $H_j$ as
solution of \eqref{eq:def-h-j}.

\section{Linear estimates} \label{sec:linear}

In this section we analyse the linearised equation \eqref{eq:u-j} and
obtain an estimate for the solution in terms of the $F_j$ containing
the forcing and lower-order terms. For this, we shall first analyse \eqref{eq:u-j}
for a fixed $j$. We will obtain a control of $H_j$ in terms of the forcing
$F_j$ and an error term $\partial_x H_j$, which will be shown to be approximately
$\jp^{\gamma} H_{j+1}$. By summing over $j$, we will find the following control.
\begin{lemma}
  \label{thm:control-h-f}
  Fix
  $m > \frac 12, \frac 12 \le \alpha \le \frac 12 + \gamma, 1 \le
  \gamma \le 2, r \in \R$. Then there exists a constant
  $\mathcal{C} = \mathcal{C}(m,\alpha,\gamma,r)$ such that for
  all $\tau_1$, $\beta$ and $T$ such that
  \begin{equation*}
    \beta \ge \mathcal{C} (1 + \| (U^P,V^P) \|_{low})\,
    (1 + \frac{1}{\tau_1} + \| (U^P,V^P) \|_{low})
    \text{ and }
    \tau(T) \ge \tau_1
  \end{equation*}
  the $H_j$'s defined by \eqref{eq:def-h-j} for  solutions $u_j$'s of
  \eqref{eq:u-j} satisfy
  \begin{equation*}
    \begin{aligned}
      &\sum_{j=0}^{\infty} \beta^2 \jp^{2\gamma}
      \left[
        \int_0^T \| H_j(t) \|_j^2 \, \dd t
        + \frac{1}{\beta\jp} \| H_j(T) \|_j^2
        + \frac{1}{\beta\jp} \int_0^T \| \partial_y H_j \|_j^2 \, \dd t
      \right] \\
      &\le 16 \sum_{j=0}^{\infty}
      \left[
        \frac{\jp^{2\gamma-4}}{\beta^2}
        \int_0^T \| F_j(t) \|_j^2\, \dd t
        + \frac{\jp^{2\gamma-3}}{\beta}
        \| u_{\init,j} \|_j^2 \right].
    \end{aligned}
  \end{equation*}
\end{lemma}
Here we use a low-order control of $U^P$ and $V^P$ in order to control
the commutator error terms. From the required bounds, we define the
low-order norm as
\begin{equation} \label{def:low}
  \begin{aligned}
    \| (U^P,V^P) \|_{low}
    = & \sup_{t\in[0,T]} \max \Big(
    \max_{0\le k \le 3}\| \pa_x^k U^P \|_{\infty}, \|\pa_x \pa^2_y U^P \|_\infty,
    \| (1+y) \partial_y U^P \|_{\infty},
    \| (1+y) \partial_y^2 U^P \|_{\infty},\\
    &\| (1+y) \partial_y U^P \|_{L^\infty_xL^2_y(\rho_0)},
    \|  \partial_{xy} U^P \|_{L^\infty_xL^2_y(\rho_0)},  \| \partial_{xxy} U^P \|_{L^\infty_xL^2_y(\rho_0)}, \\
    &    \| (1+y)^2 \partial_y^2 U^P \|_{L^\infty_xL^2_y(\rho_0)}, \| (1+y) \partial_{x} \partial_y^2 U^P \|_{L^\infty_xL^2_y(\rho_0)},\quad    \max_{0\le k \le 2}\left\| \frac{\pa_x^k V^P}{1+y} \right\|_{\infty}
    \Big).
  \end{aligned}
\end{equation}

Although a main ingredient of our proof, the unknown $H_j$ is less
natural that the usual $u_j$, notably for the future treatment of the
nonlinearity, which involves $u_j$ and $\omega_j = \partial_y
u_j$. This is why we shall we relate the control of $H_j$ to $u_j$ and
show:
\begin{proposition}
  \label{thm:final-linear-control}
  Fix
  $m > \frac 12, \frac 12 \le \alpha \le \frac 12 + \gamma, 1 \le
  \gamma \le 2, r \in \R$. Then there exist constants
  $C = C(m,\alpha,\gamma,r)$ and
  $\mathcal{C} = \mathcal{C}(m,\alpha,\gamma,r)$ such that for all
  $\tau_1$, $\beta$ and $T$ such that
  \begin{equation*}
    \beta \ge \mathcal{C} (1 + \| (U^P,V^P) \|_{low})\,
    (1 + \frac{1}{\tau_1} + \| (U^P,V^P) \|_{low})
    \text{ and }
    \tau(T) \ge \tau_1
  \end{equation*}
  the solution $u$ of \eqref{eq:u-j} satisfies
  \begin{equation*}
    \begin{aligned}
      &\int_0^T \| u \|_{\gamma,\tau,r}^2\, \dd t
      + \sup_{t\in[0,T]}
      \frac{1}{\beta} \| u \|_{\gamma,\tau,r-\frac{\gamma}{2}}^2
      + \int_0^T \frac{1}{\beta} \| (1+y) \omega \|_{\gamma,\tau,r+1-\gamma}^2 \, \dd t\\
      &\quad+ \sup_{t\in[0,T]}
      \frac{1}{\beta^2}
      \| (1+y) \omega \|_{\gamma,\tau,r+\frac 12 - \gamma}^2
      + \frac{1}{\beta^2}
      \int_0^T \| (1+y) \partial_y \omega \|_{\gamma,\tau,r+\frac 12 -\gamma}^2\, \dd t\\
      &\le C
      \left[
        \frac{1}{\beta^2}
        \sum_{j=0}^{\infty} \int_0^T \frac{1}{\jp^{4-2\gamma}}
        \| F_j \|_j^2\, \dd t
        + \frac{1}{\beta^2}
        \sum_{j=0}^{\infty} \int_0^T \frac{1}{\jp^{2\gamma-1}}
        \| (1+y) F_j \|_j^2 \, \dd t \right] \\
      &\quad+ \frac{C}{\beta^2}
      \sum_{j=0}^{\infty} \int_0^T \frac{1}{\jp^{2\gamma-1}}
      \| F_j|_{y=0} \|_{L^2_x}^2 \, \dd t
      + C
      \left[
        \frac{1}{\beta} \| u_{\init} \|_{\gamma,\tau_0,r+\gamma-\frac{3}{2}}^2
        + \frac{1}{\beta^2} \| (1+y) \omega_{\init} \|_{\gamma,\tau_0,r+\frac 12 - \gamma}^2
      \right].
    \end{aligned}
  \end{equation*}
  For $\gamma \ge 5/4$ this is
  \begin{equation*}
    \begin{aligned}
      &\int_0^T \| u \|_{\gamma,\tau,r}^2\, \dd t
      + \sup_{t\in[0,T]}
      \frac{1}{\beta} \| u \|_{\gamma,\tau,r-\frac{\gamma}{2}}^2
      + \int_0^T \frac{1}{\beta} \| (1+y) \omega \|_{\gamma,\tau,r+1-\gamma}^2 \, \dd t\\
      &\quad+ \sup_{t\in[0,T]}
      \frac{1}{\beta^2}
      \| (1+y) \omega \|_{\gamma,\tau,r+\frac 12 - \gamma}^2
      + \frac{1}{\beta^2}
      \int_0^T \| (1+y) \partial_y \omega \|_{\gamma,\tau,r+\frac 12 -\gamma}^2\, \dd t\\
      &\le C
      \left[
        \frac{1}{\beta^2}
        \sum_{j=0}^{\infty} \int_0^T \frac{1}{\jp^{4-2\gamma}}
        \left\| \left(1+\frac{y}{\jp^{2\gamma - \frac 52}}\right) F_j
        \right\|_j^2\, \dd t
      \right]\\
      &\quad+
      \frac{C}{\beta^2}
      \sum_{j=0}^{\infty} \int_0^T \frac{1}{\jp^{2\gamma-1}}
      \| F_j|_{y=0} \|_{L^2_x}^2 \, \dd t
      + C
      \left[
        \frac{1}{\beta} \| u_{\init} \|_{\gamma,\tau_0,r+\gamma-\frac{3}{2}}^2
        + \frac{1}{\beta^2} \| (1+y) \omega_{\init} \|_{\gamma,\tau_0,r+\frac 12 - \gamma}^2
      \right].
    \end{aligned}
  \end{equation*}
\end{proposition}

\subsection{Estimate for $H_j$}
We focus first on Lemma \ref{thm:control-h-f}. The idea is to  use  the solution $\phi_j$ of
\eqref{eq:def-phi-j} as a test function in  \eqref{eq:u-j}. Taking the weighted scalar product  and integrating over $[0,T]$, we find for the first term in \eqref{eq:u-j}:
\begin{equation*}
  \begin{aligned}
    &\int_0^T
    \ip{\Big(\partial_t + \beta \jp + U^P \partial_x + \jp \partial_x U^p
      + V^P \partial_y - \partial_y^2\Big) u_j}
    {\phi_j}_j
    \, \dd t \\
    &= - \ip{u_{\init,j}}{\phi_j(0)}_j\\
    &+ \int_0^T \ip{u_j}
    {\left(-\partial_t + \beta \jp - U^P \partial_x + j \partial_x U^p
        - V^P \partial_y - \partial_y V^P - V^P \frac{\partial_y \rho_j}{\rho_j}
        - \left(\partial_y + \frac{\partial_y \rho_j}{\rho_j}\right)^2
      \right) \phi_j}_j\, \dd t \\
    &= - \ip{u_{\init,j}}{\phi_j(0)}_j
    + \int_0^T \ip{u_j}{H_j}_j\, \dd t.
  \end{aligned}
\end{equation*}
Note that there is no boundary term as $u_j$ and $\phi_j$ vanish at
the boundaries.  Differentiating \eqref{eq:def-h-j}, we can replace
$u_j$ in the last integral and find
\begin{equation*}
  \begin{aligned}
    \int_0^T &\ip{u_j}{H_j}_j\, \dd t \\
    = &\int_0^T
    \ip{\Big(\partial_t + \beta \jp + U^P \partial_x + \jp \partial_x U^p
      + V^P \partial_y - \partial_y^2\Big) H_j}
    {H_j}_j\, \dd t \\
    &+ \int_0^t \ip{(\partial_y U^P \partial_x + j \partial_{xy} U^P) \int_0^y
      H_j \, \dd z}{H_j}_j\, \dd t \\
    &+ \int_0^t \ip{\partial_{xy} U^P \int_0^y H_j\, \dd z}{H_j}_j \,
    \dd t
    + \int_0^T \ip{\partial_y V^P H_j}{H_j}_j \, \dd t \\
    = & \frac{1}{2} \| H_j(T) \|_j^2
    + \beta \jp \int_0^T \| H_j(t) \|_j^2\, \dd t
    + \int_0^T \| \partial_y H_j(t) \|_j^2\, \dd t \\
    &+ \int_0^t \ip{(\partial_y U^P \partial_x + j \partial_{xy} U^P) \int_0^y
      H_j \, \dd z}{H_j}_j\, \dd t \\
    &+ \int_0^t \ip{\partial_{xy} U^P \int_0^y H_j\, \dd z}{H_j}_j \,
    \dd t
    + \int_0^T \ip{\partial_y V^P H_j}{H_j}_j \, \dd t \\
    &+ (j{+}\frac 12) \int_0^T \ip{\partial_x U^P H_j}{H_j}_j\, \dd t
    - \frac 12 \int_0^T
    \ip{\left(\partial_y V^P + V^P \frac{\partial_y
          \rho_j}{\rho_j}\right)H_j}{H_j}_j \, \dd t
    + \int_0^T \ip{\frac{\partial_y \rho_j}{\rho_j} H_j}{H_j}_j \, \dd t.
  \end{aligned}
\end{equation*}
By the boundary values of \(H_j\) there are again no boundary terms
from partial integration in \(y\).  In the last expression, the first
line contains the good controlled terms, the second line will cancel
the leading contribution from the bad terms
$\partial_y U^P v_j + j \partial_{xy} U^P \partial_x^{-1} v_j$ (see
below), while the last two lines collect the error terms.

Next, we compute the contribution from the terms with $v_j$ using
$v_j = -\partial_x \int_0^y u_j \, \dd z$:
\begin{equation*}
  \begin{aligned}
    &\int_0^T \ip{\partial_y U^P v_j}{\phi_j}_j\,\dd t \\
    &= - \int_0^T \ip{\partial_y U^P
      \partial_x \left[
        \Big(\partial_t + \beta \jp + U^P \partial_x + \jp \partial_x U^p
        + V^P \partial_y - \partial_y^2\Big) \int_0^y H_j\, \dd z
      \right]}{\phi_j}_j\,\dd t \\
    &= - \int_0^T \ip{\partial_y U^P
      \Big(\partial_t + \beta \jp + U^P \partial_x + \jp \partial_x U^p
      + V^P \partial_y - \partial_y^2\Big) \partial_x \int_0^y H_j\, \dd z
    }{\phi_j}_j\,\dd t \\
    &\quad - \int_0^T \ip{\partial_y U^P (\partial_x U^P \partial_x + \jp \partial_x^2
      U^P + \partial_x V^P \partial_y) \int_0^y H_j\, \dd z}{\phi_j}_j\,\dd t \\
    &= - \int_0^T \ip{
      \Big(\partial_t + \beta \jp + U^P \partial_x + \jp \partial_x U^p
      + V^P \partial_y - \partial_y^2\Big)
      \left[\partial_y U^P \partial_x \int_0^y H_j\, \dd z\right]
    }{\phi_j}_j\,\dd t \\
    &\quad + \int_0^T \ip{
      \Big(
      (\partial_t + U^P \partial_x + V^P \partial_y) \partial_y U^P
      - 2 \partial_y^2 U^P \partial_y - \partial_y^3 U^P
      \Big) \partial_x \int_0^y H_j\, \dd z}{\phi_j}_j\,\dd t \\
    &\quad - \int_0^T \ip{\partial_y U^P (\partial_x U^P \partial_x + \jp \partial_x^2
      U^P + \partial_x V^P \partial_y) \int_0^y H_j\, \dd z}{\phi_j}_j\,\dd t \\
    &= - \int_0^T \ip{\partial_y U^P \partial_x \int_0^y H_j\, \dd z}{H_j}_j\,\dd t \\
    &\quad + \int_0^T \ip{
      \Big(
      (\partial_t + U^P \partial_x + V^P \partial_y) \partial_y U^P
      - 2 \partial_y^2 U^P \partial_y - \partial_y^3 U^P
      \Big) \partial_x \int_0^y H_j\, \dd z}{\phi_j}_j\,\dd t \\
    &\quad - \int_0^T \ip{\partial_y U^P (\partial_x U^P \partial_x + \jp \partial_x^2
      U^P + \partial_x V^P \partial_y) \int_0^y H_j\, \dd z}{\phi_j}_j\,\dd t \\
  \end{aligned}
\end{equation*}
and
\begin{equation*}
  \begin{aligned}
    &\int_0^T\ip{j \partial_{xy} U^P \partial_x^{-1} v_j}{\phi_j}_j\,\dd t \\
    &= -j \int_0^T\ip{\partial_{xy} U^P
      \Big(\partial_t + \beta \jp + U^P \partial_x + \jp \partial_x U^p
      + V^P \partial_y - \partial_y^2\Big) \int_0^y H_j \, \dd
      z}{\phi_j}_j\,\dd t \\
    &= -j \int_0^T\ip{\partial_{xy} U^P \int_0^y H_j\, \dd z}{H_j}_j\,\dd t \\
    &\quad + j\int_0^T\ip{
      \Big(
      (\partial_t + U^P \partial_x + V^P \partial_y) \partial_{xy} U^P
      - 2 \partial_x \partial_y^2 U^P \partial_y - \partial_x \partial_y^3 U^P
      \Big) \int_0^y H_j\, \dd z}{\phi_j}_j\,\dd t.
  \end{aligned}
\end{equation*}
In both cases the leading order term cancels. Hence collecting the
terms we arrive at
\begin{equation*}
  \begin{aligned}
    &\frac{1}{2} \| H_j(T) \|_j^2
    + \beta \jp \int_0^T \| H_j(t) \|_j^2\, \dd t
    + \int_0^T \| \partial_y H_j(t) \|_j^2\, \dd t \\
    &\le \int_0^T \ip{F_j}{\phi_j}\, \dd t
    + \ip{u_{\init,j}}{\phi_j(0)}_j
    + \int_0^T \sum_{i=1}^{5} E_i\, \dd t
  \end{aligned}
\end{equation*}
where $E_1,\dots ,E_5$ collect the lower-order error terms as
\begin{align*}
  E_1 &= - \ip{\partial_{xy} U^P \int_0^y H_j\, \dd z}{H_j}_j
        - \ip{\partial_y V^P H_j}{H_j}_j, \\
  E_2 &= - (j{+}\frac 12) \ip{\partial_x U^P H_j}{H_j}_j
        + \frac 12
        \ip{\left(\partial_y V^P + V^P \frac{\partial_y
        \rho_j}{\rho_j}\right)H_j}{H_j}_j
        - \ip{\frac{\partial_y \rho_j}{\rho_j} H_j}{H_j}_j, \\
  E_3 &= - \ip{
        \Big(
        (\partial_t + U^P \partial_x + V^P \partial_y) \partial_y U^P
        - 2 \partial_y^2 U^P \partial_y - \partial_y^3 U^P
        \Big) \partial_x \int_0^y H_j\, \dd z}{\phi_j}_j, \\
  E_4 &= \ip{\partial_y U^P (\partial_x U^P \partial_x + \jp \partial_x^2
        U^P + \partial_x V^P \partial_y) \int_0^y H_j\, \dd z}
        {\phi_j}_j, \\
  E_5 &= -j \ip{
      \Big(
      (\partial_t + U^P \partial_x + V^P \partial_y) \partial_{xy} U^P
      - 2 \partial_x \partial_y^2 U^P \partial_y - \partial_x \partial_y^3 U^P
      \Big) \int_0^y H_j\, \dd z}{\phi_j}_j.
\end{align*}
Here $E_3$ and $E_4$ contain the worst terms, as they involve
$x$-derivatives of $H_j$. They are responsible for the Gevrey regularity requirement.

Assume $m \ge 0$, $\alpha \ge \frac 12$ and $\beta$ large enough so that
\cref{thm:control-phi-by-h} applies.  We can then estimate the forcing
terms as
\begin{equation*}
  \int_0^T \ip{F_j}{\phi_j}_j \, \dd t
  \le \frac{2}{\beta^3 \jp^3} \int_0^T \| F_j(t) \|_j^2 \, \dd t
  + \frac{\beta \jp}{4} \int_0^T \| H_j(t) \|_j^2\, \dd t
\end{equation*}
and
\begin{equation*}
  \ip{u_{\init,j}}{\phi_j(0)}_j
  \le \frac{2}{\beta^2 \jp^2} \| u_{\init,j} \|_j^2
  + \frac{\beta \jp}{4} \int_0^T \| H_j(t) \|_j^2\, \dd t.
\end{equation*}
Absorbing the terms with $H_j$ we therefore find
\begin{equation*}
  \begin{aligned}
    &\| H_j(T) \|_j^2
    + \beta \jp \int_0^T \| H_j(t) \|_j^2\, \dd t
    + 2 \int_0^T \| \partial_y H_j(t) \|_j^2\, \dd t \\
    &\le \frac{4}{\beta^3 \jp^3} \int_0^T \| F_j(t) \|_j^2 \, \dd t
    + \frac{4}{\beta^2 \jp^2} \| u_{\init,j} \|_j^2
    + 2 \int_0^T \sum_{i=1}^{5} E_i\, \dd t.
  \end{aligned}
\end{equation*}

We now estimate the error terms, where we repeatedly use
\cref{thm:poincare-y-rho}. For $E_1$ we find
\begin{equation*}
  E_1 \le \left[
  C_m \| \partial_{xy} U^P \|_{L^\infty_xL^2_y(\rho_0)}
  + \| \partial_y V^P \|_{\infty}
  \right] \| H_j \|_j^2.
\end{equation*}
For $E_2$ we also use \cref{thm:logarithmic-rho-bound} and assume
$\alpha \ge \frac 12$
\begin{equation*}
  E_2 \le
  \left[
    (j{+}\frac 12) \| \partial_x U^P \|_{\infty}
    + \frac 12 \| \partial_y V^P \|_{\infty}
    + C_l \jp \left(1+ \left\| \frac{V^P}{1+y} \right\|_{\infty}\right)
  \right] \| H_j \|_j^2.
\end{equation*}
In the term $E_3$ we have terms with $\partial_x H_j$, which we want
to estimate in $\| \cdot \|_{j+1}$ as they will be later controlled by
 $H_{j+1}$. Using \cref{thm:rho-j-j-1} we find
\begin{equation*}
  \begin{aligned}
    E_3 &\le
    C_m \| (1{+}y) (\partial_t + U^P \partial_x + V^P \partial_y -
    \partial_y^2) \partial_{y} U^P \|_{L^\infty_xL^2_y(\rho_0)}
    \| \partial_x H_j \|_{j+1}
    \| \phi_j \|_j \\
    &+ 2 \| (1{+}y) \partial_y^2 U^P \|_{\infty}
    \| \partial_x H_j \|_{j+1} \| \phi_j \|_j \\
    &\le 2 \| (1{+}y) \partial_y^2 U^P \|_{\infty}  \| \partial_x H_j \|_{j+1} \| \phi_j \|_j
  \end{aligned}
\end{equation*}
where we used the identity
\begin{equation} \label{eq:transport:vorticity}
(\pa_t + U^P \pa_x + V^P \pa_y) \pa_y U^P  - \pa^2_y \pa_y U^P = 0.
\end{equation}
 Similarly, we find for $E_4$ that
\begin{equation*}
  \begin{aligned}
    E_4 &\le
    C_m \| (1{+}y) \partial_y U^P \partial_x U^P \|_{L^\infty_xL^2_y(\rho_0)}
    \| \partial_x H_j \|_{j+1}
    \| \phi_j \|_j \\
    &+ \jp C_m
    \| \partial_y U^P \partial_x^2 U^P \|_{L^\infty_xL^2_y(\rho_0)}
    \| H_j \|_j \| \phi_j \|_j \\
    &+ \| \partial_y U^P \partial_x V^P \|_{\infty}
    \| H_j \|_j \| \phi_j \|
      \end{aligned}
\end{equation*}
And finally for $E_5$ we find
\begin{equation*}
  \begin{aligned}
    E_5 &\le \bigl( j C_m
    \| (\partial_t + U^P \partial_x + V^P \partial_y -
    \partial_y^2) \partial_{xy} U^P \|_{L^\infty_xL^2_y(\rho_0)} + 2j \| \partial_x \partial_y^2 U^P \|_{\infty} \bigr) \| H_j \|_j
    \| \phi_j \|_j \\
    &\le \bigl( j C_m    \|  (\pa_x U^P \partial_x + \pa_x V^P \partial_y) \partial_{y} U^P \|_{L^\infty_xL^2_y(\rho_0)} + 2j \| \partial_x \partial_y^2 U^P \|_{\infty} \bigr) \| H_j \|_j
    \| \phi_j \|_j
  \end{aligned}
\end{equation*}
where we took again advantage of \eqref{eq:transport:vorticity}.

We collect the various factors in constants  $D_1,D_2,D_3$ defined  as folllows:
\begin{equation*}
  \begin{aligned}
    D_1 & = 4    \Big(    \| (1{+}y) \partial_y^2 U^P \|_{\infty}  + C_m \| (1{+}y) \partial_y U^P \partial_x U^P
    \|_{L^\infty_xL^2_y(\rho_0)}  \Big)
  \end{aligned}
\end{equation*}
and
\begin{equation*}
  \begin{aligned}
  D_2 = 2 \Big(&\jp C_m
    \| \partial_y U^P \partial_x^2 U^P \|_{L^\infty_xL^2_y(\rho_0)}
    + \| \partial_y U^P \partial_x V^P \|_{\infty} \\
    &+ j C_m
    \|  (\pa_x U^P \partial_x + \pa_x V^P \partial_y) \partial_{y} U^P \|_{L^\infty_xL^2_y(\rho_0)}
    + 2j \| \partial_x \partial_y^2 U^p \|_{\infty}\Big)
\end{aligned}
\end{equation*}
and
\begin{equation*}
  D_3 = 2 \left(
    C_m \| \partial_{xy} U^P \|_{L^\infty_xL^2_y(\rho_0)}
    + \| \partial_y V^P \|_{\infty}
    + (j{+}\frac 12) \| \partial_x U^P \|_{\infty}
    + \frac 12 \| \partial_y V^P \|_{\infty}
    + C_l \jp \left(1+ \left\| \frac{V^P}{1+y} \right\|_{\infty}\right)
    \right).
\end{equation*}
Then
\begin{equation*}
  \begin{aligned}
    2 \int_0^T \sum_{i=1}^{5} E_i\, \dd t
    &\le D_1 \int_0^T \| \partial_x H_j \|_{j+1} \|
    \phi_j \|_j\, \dd t
    + D_2 \int_0^T \| H_j \|_j \| \phi_j \|_j\, \dd t
    + D_3 \int_0^T \| H_j \|_j^2 \, \dd t \\
    &\le \frac{1}{4} \int_0^T \beta^3 \jp^3 \| \phi_j \|_j^2\, \dd t
    + \frac{2D_1^2}{\beta^3\jp^3} \int_0^T
    \| \partial_x H_j \|_{j+1}^2 \, \dd t
    + \left(\frac{2D_2^2}{\beta^3\jp^3}+D_3\right) \int_0^T \| H_j \|_j^2 \, \dd t
  \end{aligned}
\end{equation*}
With \cref{thm:control-phi-by-h} the $\phi$ integral can be estimated
as
\begin{equation*}
  \frac{1}{4} \int_0^T \beta^3 \jp^3 \| \phi_j \|_j^2\, \dd t
  \le \frac{1}{2} \beta \jp \int_0^T \| H_j(t) \|_j^2\, \dd t
\end{equation*}
and thus can be absorbed in the LHS.

Here $\| (U^P,V^P)\|_{low}$ has been designed such that we can find
numerical constants $c_1,c_2,c_3$ such that
\begin{align*}
  D_1 &\le c_1 (1 + \| (U^P,V^P) \|_{low})^2, \\
  D_2 &\le c_2 \jp\, (1 + \| (U^P,V^P) \|_{low})^2, \\
  D_3 &\le c_3 \jp\, (\| (U^P,V^P) \|_{low}).
\end{align*}

Combining all the estimates we arrive at the following lemma.
\begin{lemma}
  \label{thm:control-h-j}
  Assume $\alpha \ge \frac 12$ and $m > \frac 12$. Then there exist a
  constant $\mathcal{C} = \mathcal{C}(m,\alpha)$ such that for
  \begin{equation*}
    \beta \ge \mathcal{C} (1 + \| (U^P,V^P) \|_{U,low})
  \end{equation*}
  and $j \in \N$ the $H_j$ defined by~\eqref{eq:def-h-j} for a
  solution $u_j$ of \eqref{eq:u} satisfy
  \begin{equation*}
    \begin{aligned}
    &2 \| H_j(T) \|_j^2
    + \beta \jp \int_0^T \| H_j(t) \|_j^2\, \dd t
    + 4 \int_0^T \| \partial_y H_j(t) \|_j^2\, \dd t \\
    &\le \frac{8}{\beta^3 \jp^3} \int_0^T \| F_j(t) \|_j^2 \, \dd t
    + \frac{8}{\beta^2 \jp^2} \| u_{\init,j} \|_j^2 \\
    &\quad+ \frac{4c_1^2(1 + \| (U^P,V^P) \|_{low})^4}{\beta^3 \jp^3}
    \int_0^T \| \partial_x H_j \|_{j+1}^2 \, \dd t.
    \end{aligned}
  \end{equation*}
\end{lemma}
\begin{proof}
  Use the previous estimates. Note that the condition on $\beta$ also
  implies that the hypothesis of \cref{thm:control-phi-by-h} is
  satisfied by choosing $\mathcal{C}$ large enough.
\end{proof}

\subsection{Relating $\partial_x H_j$ with $H_{j+1}$}
To conclude the proof of Lemma \ref{thm:control-h-f}, that will be
achieved by summation of the previous estimate over $j$, we need first
to control $\partial_x H_j$ by $H_{j+1}$.
\begin{lemma}
  \label{thm:relate-partial-h-j}
  Let $m > \frac 12$ and $\alpha \ge \frac 12$. Then there exist
  constants $\mathcal{C} = \mathcal{C}(m,\alpha)$ and
  $C = C(m,\alpha,r)$ such that for all $\tau_1$, $\beta$ and $T$ with
  \begin{equation*}
    \beta \ge \mathcal{C}
    \left(  1 + \| (U^P,V^P) \|_{low}  \right)^2, \quad \tau(T) \ge \tau_1,
  \end{equation*}
  it holds that
  \begin{equation*}
  \begin{aligned}
   &\int_0^T \| \pa_x H_j  \|_{j+1}^2 \dd t \\
   &  \le C \frac{\jp^{2\gamma}}{\tau_1^2}  \int_0^T
   \|H_{j+1}\|^2_{j+1} \, \dd t
    + C \frac{\jp^{2\alpha-2}}{\beta}  \int_0^T \|\pa_y H_j\|_{j}^2
    \, \dd t
    + \frac{C}{\beta} \int_0^T \| H_j \|_{j}^2 \, \dd t.
  \end{aligned}
  \end{equation*}
\end{lemma}
\begin{proof}
  From the definition of $u_j$, it holds that
  $\partial_x u_j(t) = \left( \frac{j+2}{j+1}\right)^r
  \frac{(j+1)^\gamma}{\tau(t)} u_{j+1}(t)$. Hence we anticipate that
  \begin{equation*}
    \partial_x H_j(t) \approx \left( \frac{j{+}2}{j{+}1}\right)^r
    \frac{\jp^\gamma}{\tau(t)} H_{j+1}(t).
  \end{equation*}
  Therefore we estimate the difference
  \begin{equation*}
    \Delta_j := \partial_x H_j - \left( \frac{j{+}2}{j{+}1}\right)^r
    \frac{\jp^\gamma}{\tau(t)} H_{j+1}.
  \end{equation*}
  From equation \eqref{eq:def-h-j} (used with indices $j$ and $j{+}1$), we
  find that
  \begin{equation} \label{eq:Delta_j}
    \Big(\partial_t + \beta \jp + U^P \partial_x + (j{+}2) \partial_x U^P
    + V^P \partial_y - \partial_y^2\Big)
    \int_0^y \Delta_j \, \dd z
    = - \left[ \jp \partial_{xx} U^P + \partial_x V^P \partial_y \right]
    \int_0^y H_j\, \dd z.
  \end{equation}
  We stress that $\int_0^y \Delta_j\, \dd z$ does not converge to zero
  at infinity, so that one can not perform $L^2$ estimates on this
  quantity. However, we can notice by \cref{thm:poincare-y-rho} that
  \begin{equation} \label{ineq:delta_j}
    \begin{aligned}
      \| \bigl(\frac{\rho_{j+1}}{\rho_0}\bigr)^{1/2} \int_0^y \Delta_j\,
      \dd z\|_{L^2}
      &\le \|\bigl(\frac{\rho_{1}}{\rho_0}\bigr)^{1/2} \|_{L^2_y} \:
      \| \bigl(\frac{\rho_{j+1}}{\rho_1}\bigr)^{1/2} \int_0^y \Delta_j
      \,\dd z \|_{L^2_x L^\infty_y} \\
      &\le C_{m-1} \| \Delta_j \|_{j+1} \\
      &< +\infty.
    \end{aligned}
  \end{equation}
  The square integrable quantity
  $\delta_j = \bigl(\frac{\rho_{j+1}}{\rho_0}\bigr)^{1/2} \int_0^y
  \Delta_j\, \dd z$ satisfies the equation
  \begin{equation}
    \label{eq:delta_j}
    \begin{aligned}
      & \Big(\partial_t + \beta \jp + U^P \partial_x + (j{+}2) \partial_x U^P
      + V^P \partial_y - \partial_y^2\Big) \delta_j \\
      = &  - \jp \partial_{xx} U^P
      \bigl(\frac{\rho_{j+1}}{\rho_0}\bigr)^{1/2}  \int_0^y H_j\,\dd z  - \partial_x V^P \bigl(\frac{\rho_{j+1}}{\rho_0}\bigr)^{1/2}   H_j  \\
      & +  V^P \pa_y  \Bigl( \bigl(\frac{\rho_{j+1}}{\rho_0}\bigr)^{1/2}\Bigr)  \bigl(\frac{\rho_{j+1}}{\rho_0}\bigr)^{-1/2} \delta_j - 2 \pa_y  \Bigl( \bigl(\frac{\rho_{j+1}}{\rho_0}\bigr)^{1/2} \Bigr) \Delta_j  -    \pa^2_y  \Bigl( \bigl(\frac{\rho_{j+1}}{\rho_0}\bigr)^{1/2} \Bigr) \bigl(\frac{\rho_{j+1}}{\rho_0}\bigr)^{-1/2} \delta_j.
    \end{aligned}
  \end{equation}
  As in \eqref{ineq:delta_j}, we obtain
  \begin{equation}
    \| \jp \partial_{xx} U^P
    \bigl(\frac{\rho_{j+1}}{\rho_0}\bigr)^{1/2}  \int_0^y H_j \, \dd z\|_{L^2}
    \le C_{m-1}  \jp \|\partial_{xx} U^P \|_{\infty} \|H_j \|_{j+1}
  \end{equation}
  We also get
  $$ \|\partial_x V^P \bigl(\frac{\rho_{j+1}}{\rho_0}\bigr)^{1/2}   H_j  \|_{L^2}\le \|\frac{1}{1+y}\partial_x V^P \|_{\infty} \|H_j \|_{j+1}. $$
  By  \cref{thm:logarithmic-rho-bound}, we find
  \begin{align*}
    \|V^P \pa_y  \Bigl( \bigl(\frac{\rho_{j+1}}{\rho_0}\bigr)^{1/2}\Bigr)  \bigl(\frac{\rho_{j+1}}{\rho_0}\bigr)^{-1/2}  \delta_j \|_{L^2} \le \|\frac{1}{1+y}V^P \|_{\infty} C_l (j+1) \|\delta_j \|_{L^2}.
  \end{align*}
  Using again  \cref{thm:logarithmic-rho-bound} and the identity
  $$ \bigl(\frac{\rho_{j+1}}{\rho_0}\bigr)^{1/2}  \Delta_j  = \pa_y \delta_j - \pa_y \Bigl( \bigl(\frac{\rho_{j+1}}{\rho_0}\bigr)^{1/2} \Bigr)  \bigl(\frac{\rho_{j+1}}{\rho_0}\bigr)^{-1/2}  \delta_j $$
  and defining
  \begin{equation} \label{def:A:j:alpha}
    A_{j,\alpha} = \max(\jp^{1-\alpha}, \log \jp, 1)
  \end{equation}
  we obtain
  \begin{equation*}
    \|2 \pa_y  \Bigl( \bigl(\frac{\rho_{j+1}}{\rho_0}\bigr)^{1/2} \Bigr) \Delta_j \|_{L^2} \le   2 C_l A_{j,\alpha} \|\pa_y \delta _j \|_{L^2} + 2 C_l^2 A_{j,\alpha}^2 \|\delta_j \|_{L^2}.
  \end{equation*}
  Eventually,
  \begin{align*}
    &  \| \pa^2_y  \Bigl( \bigl(\frac{\rho_{j+1}}{\rho_0}\bigr)^{1/2} \Bigr) \bigl(\frac{\rho_{j+1}}{\rho_0}\bigr)^{-1/2} \delta_j   \|_{L^2} \\
   &=  \| \pa_y  \Bigl( \sum_{k=1}^{j+1}\frac{1}{k^\alpha (1+\frac{y}{k^\alpha})}  \bigl(\frac{\rho_{j+1}}{\rho_0}\bigr)^{1/2} \Bigr)  \bigl(\frac{\rho_{j+1}}{\rho_0}\bigr)^{-1/2}  \delta_j \|_{L^2} \le   C A_{j,\alpha}^2   \|\delta_j \|_{L^2}
  \end{align*}
  for some constant $C = C(\alpha)$. The previous bounds combined with an energy estimate yield that for $\mathcal{C}$ large enough (we remind that $\alpha \ge \frac{1}{2}$):
  \begin{equation} \label{bound:delta_j}
    \|\delta_j(T)\|_{L^2}^2 + \beta (j+1) \int_0^T \|\delta_j
    \|_{L^2}^2\, \dd t
    + \int_0^T \|\pa_y \delta_j \|_{L^2}^2\, \dd t
    \le \jp \int_0^T \|H_j \|_{j+1}^2 \, \dd t.
  \end{equation}
  We can then take the $x$-derivative of equation \eqref{eq:delta_j} and proceed as above. For $\mathcal{C}$ large enough, we get
  \begin{align*}
    & \|\pa_x \delta_j(T)\|_{L^2}^2
      + \beta \jp \int_0^T \|\pa_x \delta_j \|_{L^2}^2\,\dd t
      + \int_0^T \|\pa_x \pa_y \delta_j \|_{L^2}^2 \,\dd t\\
    & \le \jp \int_0^T (\|\pa_x H_j \|_{j+1}^2 +\| H_j \|_{j+1}^2)\,\dd t \\
    & + \int_0^T  \left( 2 \|\pa_x  V^P  \pa_y  \Bigl( \bigl(\frac{\rho_{j+1}}{\rho_0}\bigr)^{1/2}\Bigr)  \bigl(\frac{\rho_{j+1}}{\rho_0}\bigr)^{-1/2}  \delta_j \|_{L^2}  + 2 (j{+}2) \|\pa^2_x U^P \delta_j \|_{L^2}  + 2 \|\pa_x V^P \pa_y \delta_j \|_{L^2} \right) \|\pa_x \delta_j \|_{L^2}\,\dd t
  \end{align*}
  We then use that
  $$ \|\pa_x  V^P  \pa_y  \Bigl( \bigl(\frac{\rho_{j+1}}{\rho_0}\bigr)^{1/2}\Bigr)  \bigl(\frac{\rho_{j+1}}{\rho_0}\bigr)^{-1/2}  \delta_j \|_{L^2} \le \left\|\frac{\pa_x V^P}{1+y} \right\|_{\infty} C_l \jp \|\delta_j \|_{L^2}, $$
  and
  \begin{align*}
    \|\pa_x V^P \pa_y \delta_j \|_{L^2}
    \le
    \left\|\frac{\pa_x V^P}{1+y} \right\|_{\infty} C_l \jp \|\delta_j
    \|_{L^2}
    + \left\|\frac{\pa_x V^P}{1+y} \right\|_{\infty} \|\Delta_j \|_{j+1}
  \end{align*}
  and the bound \eqref{bound:delta_j} to end up with
  \begin{equation}  \label{bound:delta_jx}
    \begin{aligned}
      & \|\pa_x \delta_j(T)\|_{L^2}^2
      + \frac{\beta \jp}{2} \int_0^T \|\pa_x \delta_j \|_{L^2}^2\, \dd t
      + \int_0^T \|\pa_x \pa_y \delta_j \|_{L^2}^2 \, \dd t\\
      & \le 2\jp \int_0^T ( \|\pa_x H_j \|_{j+1}^2 +  \| H_j \|_{j+1}^2   +  \|\Delta_j \|_{j+1}^2) \, \dd t.
    \end{aligned}
  \end{equation}
  To estimate directly $\Delta_j$, we differentiate the equation
  \eqref{eq:Delta_j} with respect to $y$, which gives
  \begin{align*}
    &  \Big(\partial_t + \beta \jp + U^P \partial_x + (j{+}2) \partial_x U^p
      + V^P \partial_y + \partial_y V^P - \partial_y^2\Big) \Delta_j  \\
    = & -\jp\pa_{xx} U^P  H_j - \jp\pa_{xxy} U^P \int_0^y  H_j  - \pa_{xy} V^P H_j  - \pa_x V^P \pa_y H_j  \\
    & -  \pa_y U^P \pa_x \int_0^y  \Delta_j - (j{+}2) \pa_{xy} U^P  \int_0^y  \Delta_j.
  \end{align*}

  We take the $\ip{\,}{}_{j+1}$ scalar product with $\Delta_j$:
  \begin{equation*}
    \begin{aligned}
      &\left(\frac 12 \partial_t + \beta \jp\right)
      \left\| \Delta_j \right\|_{j+1}^2 - \Bigl[ \left(j{+}2\right) \| \partial_x U^P \|_{\infty}
      + \frac 12  \| \partial_y V^P \|_{\infty}
      + \frac 12  \| V^P \frac{\partial_y \rho_{j+1}}{\rho_{j+1}} \|_{\infty} \Bigr]
      \left\| \Delta_j  \right\|_{j+1}^2 \\
     &\quad + \| \pa_y \Delta_j \|_{j+1}^2
      - \ip{\pa_y \Delta_j}{\frac{\partial_y \rho_{j+1}}{\rho_{j+1}}  \Delta_j}_{j+1} \\
      &\le \jp  \bigl(\|\pa_{xx} U^P \|_{\infty} +   C_m \| \partial_{xxy} U^P \|_{L^\infty_xL^2_y(\rho_0)} + \|\pa_{xy} V^P \|_{\infty}\bigr)
      \| H_j \|_{j+1} \left\|  \Delta_j  \right\|_{j+1} \\
      &\quad+ \left\| \frac{\pa_x V^P}{1+y} \right\|_{\infty}
      \| (1{+}y) \pa_y H_j \|_{j+1} \left\|  \Delta_j \right\|_{j+1} \\
      &\quad+ \left( \|\pa_{y} U^P \sqrt{\rho_0} \|_{\infty} \| \pa_x \delta_j \|_{L^2}
        + (j{+}2) \|\pa_{xy} U^P \sqrt{\rho_0}\|_{\infty} \| \delta_j \|_{L^2} \right) \left\|  \Delta_j \right\|_{j+1} .
    \end{aligned}
  \end{equation*}
  By the 1d Sobolev imbedding theorem, we find that for a constant
  $C = C(m)$ it holds that
  \begin{equation*}
    \|\pa_{y} U^P \sqrt{\rho_0} \|_{\infty}
    \le C \|(U^P,V^P)\|_{low}
    \text{ and }
    \|\pa_{xy} U^P \sqrt{\rho_0}\|_{\infty}
    \le C \|(U^P,V^P)\|_{low}.
  \end{equation*}
  Combining these last two inequalities with \eqref{bound:delta_j},
  \eqref{bound:delta_jx} and the inequality
  $$\|(1+y)\pa_y H_j\|_{j+1} \le \jp^\alpha \|\pa_y H_j \|_j,$$
  and taking $\mathcal{C}$ large enough, we obtain
  \begin{equation}
    \begin{aligned}
      & \|\Delta_j(T) \|_{j+1}^2
      + \beta \jp \int_0^T \|\Delta_j\|_{j+1}^2 \, \dd t
      +  \int_0^T \|\pa_y \Delta_j\|_{j+1}^2 \, \dd t\\
      &\le \jp \int_0^T \|H_j\|_{j+1}^2  \, \dd t
      +  \jp^{2\alpha-1}  \int_0^T \|\pa_y H_j\|_{j}^2\, \dd t
      + \frac{1}{\beta\jp} \int_0^T \| \partial_x H_j \|_{j+1}^2 \, \dd t.
    \end{aligned}
  \end{equation}
  \Cref{thm:relate-partial-h-j} follows straightforwardly.
\end{proof}

Combining \cref{thm:control-h-j,thm:relate-partial-h-j}, we will now
prove \cref{thm:control-h-f}.
\begin{proof}[Proof of \cref{thm:control-h-f}]
  We choose $\mathcal{C}$ such that
  \cref{thm:control-h-j,thm:relate-partial-h-j} apply. We multiply the
  inequality in \cref{thm:control-h-j} by $\beta \jp^{2\gamma-1}$ and
  sum over $j$ to get
  \begin{equation*}
    \begin{aligned}
      &\sum_{j=0}^\infty \beta^2 \jp^{2\gamma}
      \left[
        \int_0^T \| H_j(t) \|_j^2\, \dd t
        + \frac{1}{\beta\jp} \| H_j(T) \|_j^2
        + \frac{1}{\beta\jp} \int_0^T \| \partial_y H_j(t) \|_j^2\, \dd t
      \right] \\
      &\le
      8 \sum_{j=0}^{\infty}
      \left[
        \frac{\jp^{2\gamma-4}}{\beta^2}
        \int_0^T \| F_j(t) \|_j^2\, \dd t
        + \frac{\jp^{2\gamma-3}}{\beta}
        \| u_{\init,j} \|_j^2
      \right]\\
      &\quad+ \sum_{j=0}^{\infty}
      \frac{4c_1^2(1 + \| (U^P,V^P) \|_{low})^4}{\beta^2}
      \jp^{2\gamma-4}
      \int_0^T \| \partial_x H_j \|_{j+1}^2 \, \dd t.
    \end{aligned}
  \end{equation*}

  Taking $\mathcal{C}$ large enough, we can then find by
  \cref{thm:relate-partial-h-j} a constant $C = C(m,\alpha,r)$ such
  that
  \begin{equation*}
    \begin{aligned}
      &\sum_{j=0}^{\infty} \jp^{2\gamma-4} \int_0^T \| \partial_x H_j
      \|_{j+1}^2\, \dd t \\
      &\le C \left(1+\frac{1}{\tau^2}\right)   \sum_{j=0}^\infty
      \jp^{4\gamma-4} \int_0^T \| H_j \|_j^2\, \dd t
      + \frac{C}{\beta}  \sum_{j=0}^\infty  \jp^{2(\gamma+\alpha)-6} \int_0^T \|\pa_y H_j \|_j^2\, \dd t \\
      & \le  C \left(1+\frac{1}{\tau^2}\right) \sum_{j=0}^\infty  \jp^{4\gamma-4}
      \left[\int_0^T \| H_j \|_j^2  \, \dd t+ \frac{1}{\beta \jp} \int_0^T \| \pa_y H_j \|_j^2\, \dd t \right].
    \end{aligned}
  \end{equation*}
  We have used here that $\alpha \le \gamma +\frac{1}{2}$.  Hence, the last term at the right-hand side can be absorbed if
  \begin{equation*}
    \frac 12 \beta^2 \jp^{2\gamma}
    \ge \frac{4Cc_1^2(1 + \| (U^P,V^P) \|_{low})^4}{\beta^2}
    \left(1+\frac{1}{\tau^2}\right)   \jp^{4\gamma-4},
  \end{equation*}
  which can be ensured by a suitable large $\mathcal{C}$ if
  $\gamma \le 2$.
\end{proof}

\subsection{Control of $u_j$ and $\omega_j$}
We now relate the estimates on $H_j$ to $u_j$ and start with an
estimate for the $L^2$ norm.
\begin{lemma}
  \label{thm:control-u-by-h-l2}
  Let $m > \frac 12$ and $\alpha \ge \frac 12$. Then there exists a
  constant $\mathcal{C} = \mathcal{C}(m,\alpha)$ such that for
  \begin{equation*}
    \beta \ge \mathcal{C}
    \left(
      1 + \| (U^P,V^P) \|_{low}
    \right)
  \end{equation*}
  and for any $\epsilon_1, \epsilon_2, \epsilon_3, \epsilon_4 > 0$ it holds that
  \begin{equation*}
    \begin{aligned}
      &\frac 12 \int_0^T \| u_j \|_j^2\, \dd t
      - \frac{\epsilon_1}{\jp^{2\gamma}}
      \int_0^T \| \partial_x u_j \|_{j+1}^2\, \dd t
      - \frac{\epsilon_2}{\beta\jp^{\gamma}}
      \int_0^T \| \partial_y u_j \|_j^2\, \dd t\\
      &\quad - \frac{\epsilon_3}{4\beta \jp^\gamma} \| u_j (T) \|_j^2
      - \frac{\epsilon_4}{\beta^2 \jp^{2\gamma}} \int_0^T \|
        \partial_y^2 u_j(t) \|_j^2\, \dd t\\
      &\le
      \frac{\beta \jp^\gamma}{\epsilon_3} \| H_j(T) \|_j^2 \\
      &\quad + \left[
        16 \beta^2 \jp^2
        + \frac{\jp^{2\gamma}}{\epsilon_1} C_m^2
        \| (1{+}y) \partial_y U^P \|_{L^\infty_{t,x}L^2_y(\rho_0)}^2
        + \frac{\beta^2\jp^{2\gamma}}{4\epsilon_4}
      \right] \int_0^T \| H_j \|_j^2 \, \dd t \\
      &\quad+
      \left[
        \frac{\beta\jp^{\gamma}}{4\epsilon_2} + 16 C_l^2 A_{j,\alpha}^2
      \right] \int_0^T \| \partial_y H_j \|_j^2 \, \dd t
      + \int_0^T \| H_j \|_j \| F_j \|_j\, \dd t
    \end{aligned}
  \end{equation*}
  where $u_j$ is satisfying \eqref{eq:u-j}, $A_{j,\alpha}$ is defined in \eqref{def:A:j:alpha} and $H_j$ is defined by
  \eqref{eq:def-h-j}.
\end{lemma}
\begin{proof}
  Using the definition \eqref{eq:def-h-j} of $H_j$ we find
  \begin{equation}
    \label{eq:proof:u-by-h-l2:u}
    \begin{aligned}
      \int_0^T \| u_j(t) \|_j^2 \, \dd t
      = &\int_0^T
      \ip{\Big(\partial_t + \beta \jp + U^P \partial_x + \jp \partial_x U^p
        + V^P \partial_y + \partial_y V^P - \partial_y^2\Big) H_j}{u_j}_j\, \dd t \\
      &+ \int_0^T \ip{ (\partial_y U^P \partial_x + \jp \partial_{xy}
        U^P) \int_0^y H_j\, \dd z}{u_j}_j\, \dd t.
    \end{aligned}
  \end{equation}
  By the evolution equation \eqref{eq:u-j} for $u_j$, the
  first term can be written (from the partial integration in \(y\)
  there is no boundary term as \(u|_{y=0} = 0\))
  \begin{equation*}
    \begin{aligned}
      &\int_0^T
      \ip{\Big(\partial_t + \beta \jp + U^P \partial_x + \jp \partial_x U^p
        + V^P \partial_y + \partial_y V^P - \partial_y^2\Big) H_j}{u_j}_j\, \dd t \\
      &= \ip{H_j(T)}{u_j(T)}_j
      + \int_0^T
      \ip{H_j}{ (-\partial_t + \beta \jp - U^P\partial_x + j \partial_x
        U^P - V^P \partial_y - V^P \frac{\partial_y \rho_j}{\rho_j})
        u_j}\\
      &\quad+ \int_0^T \ip{\partial_y H_j}{(\partial_y + \frac{\partial_y
          \rho_j}{\rho_j}) u_j}_j \\
      &= \ip{H_j(T)}{u_j(T)}_j
      + \int_0^T
      \ip{H_j}{ (2\beta \jp + (2j{+}1) \partial_x U^P - V^P \frac{\partial_y
          \rho_j}{\rho_j}) u_j}_j\, \dd t \\
      &\quad+ \int_0^T
      \ip{H_j}{\partial_y U^P v_j + j \partial_{xy} U^P \partial_x^{-1} v_j}_j\, \dd t
      - \int_0^T \ip{H_j}{F_j}_j\, \dd t \\
      &\quad+ \int_0^T \ip{\partial_y H_j}{(\partial_y + \frac{\partial_y
          \rho_j}{\rho_j}) u_j}_j\, \dd t
      - \int_0^T \ip{H_j}{\partial_y^2 u_j}_j\, \dd t.
    \end{aligned}
  \end{equation*}
  The terms can now be bounded using
  \cref{thm:logarithmic-rho-bound}:
  \begin{equation*}
    \begin{aligned}
      &\ip{H_j}{ (2\beta \jp + (2j{+}1) \partial_x U^P - V^P \frac{\partial_y
          \rho_j}{\rho_j}) u_j}_j \\
      &\le
      \left\|
        2\beta \jp + (2j{+}1) \partial_x U^P - V^P \frac{\partial_y \rho_j}{\rho_j}
      \right\|_{\infty}
      \| H_j \|_j \| u_j \|_j \\
      & \le \jp
      \left[
        2\beta + 2 \| \partial_x U^P \|_{\infty}
        + C_l \left\| \frac{V^P}{1+y} \right\|_{\infty}
      \right]
      \| H_j \|_j \| u_j \|_j.
    \end{aligned}
  \end{equation*}
  Recalling that $v_j = -\partial_x \int_0^y u_j \, \dd z$ we find
  \begin{equation*}
    \begin{aligned}
      &\ip{H_j}{\partial_y U^P v_j + j \partial_{xy} U^P
        \partial_x^{-1} v_j}_j\, \dd t \\
      &\le C_m \| (1{+}y) \partial_y U^P \|_{L^\infty_xL^2_y(\rho_0)}
      \| H_j \|_{j} \| \partial_x u_j \|_{j+1}
      + j \| \partial_{xy} U^P \|_{L^\infty_xL^2_y(\rho_0)}
      \| H_j \|_j \| u_j \|_j.
    \end{aligned}
  \end{equation*}
  For the forcing terms we find
  \begin{equation*}
    - \ip{H_j}{F_j}_j
    \le \| H_j \|_j \| F_j \|_j.
  \end{equation*}
  The diffusion terms give
  \begin{equation*}
    \begin{aligned}
      &   \ip{\partial_y H_j}{(\partial_y + \frac{\partial_y
          \rho_j}{\rho_j}) u_j}_j
      -  \ip{H_j}{\partial_y^2 u_j}_j \\
      \le &
      \| \partial_y H_j \|_j\, \| \partial_y u_j \|_j
      + C_l A_{j,\alpha} \| \partial_y H_j \|_j\, \| u_j \|_j
      + \| H_j \|_j \| \partial_y^2 u_j \|_j.
    \end{aligned}
  \end{equation*}
  The integrand in the second integral in \eqref{eq:proof:u-by-h-l2:u}
  can be estimated as
  \begin{equation*}
    \begin{aligned}
      &\ip{ (\partial_y U^P \partial_x + \jp \partial_{xy}
        U^P) \int_0^y H_j\, \dd z}{u_j}_j\\
      &\le C_m \| (1{+}y) \partial_y U^P \|_{L^\infty_xL^2_y(\rho_0)}
      \| H_j \|_j \| \partial_x u_j \|_{j+1}
      + j C_m \| \partial_{xy} U^P \|_{L^\infty_xL^2_y(\rho_0)} \|H \|_j \| u \|_j.
    \end{aligned}
  \end{equation*}
  Collecting the terms we find by choosing $\mathcal{C}$ large enough that
  \begin{equation*}
    \begin{aligned}
      \int_0^T \| u_j \|^2\, \dd t
      &\le \ip{H_j(T)}{u_j(T)}_j
      + 4 \beta \jp \int_0^T \| H_j \|_j \| u \|_j \, \dd t\\
      &\quad+ 2C_m \| (1{+}y) \partial_y U^P \|_{L^\infty_xL^2_y(\rho_0)}
      \int_0^T \| H_j \|_j \| \partial_x u_j \|_{j+1}\, \dd t \\
      &\quad+ \int_0^T \Big(
      \| \partial_y H_j \|_j\, \| \partial_y u_j \|_j
      + C_l A_{j,\alpha} \| \partial_y H_j \|_j\, \| u_j \|_j
      + \| H_j \|_j \| \partial_y^2 u_j \|_j
      \Big)\, \dd t \\
      &\quad+ \int_0^T \| H_j \|_j \| F_j \|_j \, \dd t.
    \end{aligned}
  \end{equation*}
  Splitting the products gives the claimed estimate.
\end{proof}

The missing terms can be estimated by the evolution of \(u_j\) and
\(\omega_j = \partial_y u_j\). For \(u_j\) we find:
\begin{lemma}
  \label{thm:control-u-by-h-linf}
  Let $m > \frac 12$ and $\alpha \ge \frac 12$. Then there exists a
  constant $\mathcal{C} = \mathcal{C}(m,\alpha)$ such that for
  \begin{equation*}
    \beta \ge \mathcal{C}
    \left(
      1 + \| (U^P,V^P) \|_{low}
    \right)
  \end{equation*}
  the solution $u_j$ of \eqref{eq:u-j} satisfies
  \begin{equation*}
    \begin{aligned}
      &\frac{1}{2} \| u_j(T) \|_j^2
      + \frac 12 \int_0^T \| \partial_y u_j \|^2\, \dd t\\
      &\quad-4 \beta \jp^\gamma \int_0^T \| u_j \|_j^2\, \dd t
      - \frac{C_m^2 \| (1{+}y) \partial_y U^P
        \|_{L^\infty_xL^2_y(\rho_0)}^2}{\beta \jp^\gamma}
      \int_0^T \| \partial_x u_{j} \|_{j+1}^2 \, \dd t\\
      &\le \frac{1}{2} \| u_{\init,j} \|_j^2
      + \frac{1}{\beta \jp^\gamma} \int_0^T \| F_j \|_j^2\, \dd t.
    \end{aligned}
  \end{equation*}
\end{lemma}
\begin{proof}
  By \eqref{eq:u-j} we find
  \begin{equation*}
    \begin{aligned}
      \ip{\partial_t u_j}{u_j}_j
      &= \ip{\Big(-\beta \jp - U^P \partial_x - \jp \partial_x U^p
        - V^P \partial_y + \partial_y^2\Big)u_j}{u_j}_j\\
      &\quad- \ip{\partial_y U^P v_j + j \partial_{xy} U^P \partial_x^{-1}
        v_j}{u_j}_j
      + \ip{F_j}{u_j}_j \\
      &\le - \frac 12 \| \partial_y u_j \|_j^2
      + 4 \beta \jp^\gamma \| u_j \|_j^2
      + \frac{C_m^2 \| (1{+}y) \partial_y U^P
        \|_{L^\infty_xL^2_y(\rho_0)}^2}{\beta \jp^\gamma}
      \| \partial_x u_{j+1} \|_{j+1}^2 \\
      &\quad+ \frac{1}{\beta \jp^\gamma} \| F_j \|_j^2,
    \end{aligned}
  \end{equation*}
  where there is no boundary term from the partial integration in
  \(y\) as \(u_j\) vanishes at the boundary and we used in the
  inequality that $\mathcal{C}$ can be chosen large enough.
  Integrating this over $[0,T]$ gives the claimed result.
\end{proof}

By differentiating \eqref{eq:u-j} in $y$ and find
\begin{equation}
  \label{eq:omega-j}
  \Big(\partial_t + \beta \jp + U^P \partial_x + \jp \partial_x U^p
  + V^P \partial_y + \partial_y V^P - \partial_y^2\Big) \omega_j
  + \partial_{yy} U^P v_j + j \partial_{xyy} U^P \partial_x^{-1} v_j
  + \partial_{xy} U^P u_j
  = \partial_y F_j.
\end{equation}
This immediately yields the following control for \(\omega_j\).
\begin{lemma}
  \label{thm:vorticity-control}
  Let $m > \frac 12$ and $\alpha \ge \frac 12$. Then there exists a
  constant $\mathcal{C} = \mathcal{C}(m,\alpha)$ such that for
  \begin{equation*}
    \beta \ge \mathcal{C}
    \left(
      1 + \| (U^P,V^P) \|_{low}
    \right)
  \end{equation*}
  the vorticity $\omega_j = \partial_y u_j$ satisfies
  \begin{equation*}
    \begin{aligned}
      & \| (1+y) \omega_j(T) \|_j^2
      + \int_0^T \beta \jp \| (1+y) \omega_j \|_j^2\, \dd t
      + \int_0^T \| (1+y) \partial_y \omega_j \|_j^2\, \dd t\\
      &\le
      \frac{4C_m^2 \| (1+y)^2 \partial_{yy} U^P
        \|_{L^\infty_{t,x}L^2_y(\rho_0)}^2}{\beta\jp}
      \int_0^T \| \partial_x u_{j} \|_{j+1}^2\, \dd t \\
      & +
      \frac{4C_m^2 \jp \| (1+y) \partial_{xyy} U^P \|_{L^\infty_{t,x}L^2_y(\rho_0)}^2}{\beta}
      \int_0^T \| u_j \|_j^2\, \dd t\\
      & + \| (1+y) \omega_{\init,j} \|_j^2
      + 4 \int_0^T \| (1+y) F_j \|_j^2\, \dd t
      + 4 \int_0^T \| F_j|_{y=0} \|_{L^2_x}^2\, \dd t.
    \end{aligned}
  \end{equation*}
\end{lemma}
\begin{proof}
  Integrate \eqref{eq:omega-j} against $(1+y)^2\omega_j$ in $\| \cdot \|_j$. This
  yields
  \begin{equation*}
    \begin{aligned}
      &\frac 12 \partial_t \| (1+y) \omega_j \|_j^2
      + \beta \jp \| (1+y) \omega_j \|_j^2
      + \| (1+y) \partial_y \omega_j \|_j^2 \\
      &\le j \| \partial_x U^P \|_{\infty}
      \| (1+y) \omega_j \|_j^2 + \left\| V^P \left(\frac{\partial_y \rho_j}{\rho_j} +
          \frac{\partial_y(1+y)^2}{(1+y)^2}\right)
      \right\|_{\infty} \| (1+y) \omega_j \|_j^2 \\
      &\quad+ \| \omega_j|_{y=0} \|_{L^2_x} \| \partial_y \omega_j|_{y=0} \|_{L^2_x}\\
      &\quad+ \| (1+y) \partial_y \omega_j \|_j
        \left\| (1+y) \left(\frac{\partial_y \rho_j}{\rho_j} +
            \frac{\partial_y(1+y)^2}{(1+y)^2}\right)
          \omega_j \right\|_j \\
      &\quad+ C_m \| (1+y)^2 \partial_{yy} U^P \|_{L^\infty_xL^2_y(\rho_0)}
      \| \partial_x u_{j} \|_{j+1} \| (1+y) \omega_j \|_j\\
      &\quad+ j C_m \| (1+y) \partial_{xyy} U^P \|_{L^\infty_xL^2_y(\rho_0)}
      \| u_{j} \|_{j} \| (1+y) \omega_j \|_j\\
      &\quad+ \| (1+y) F_j \|_j\,
      \left\| (1+y) \left(\partial_y + \frac{\partial_y
            \rho_j}{\rho_j}
        + \frac{\partial_y (1+y)^2}{(1+y)^2}\right) \omega_j \right\|_j,
    \end{aligned}
  \end{equation*}
  where we find a boundary term from the diffusion and there is no
  boundary term from \(V^P\partial_y\) because \(V^P|_{y=0}=0\).

  From \eqref{eq:u-j} we find \(\partial_y \omega_j|_{y=0} =
  F_j|_{y=0}\). For \(\omega_j|_{y=0}\) write
  \begin{equation*}
    |\omega_j(y=0)| \le
    \int_0^1
    \left[
      \omega_j \sqrt{\rho_j}
      + \int_0^y |(\omega_j \sqrt{\rho_j})'|\, \dd z
    \right]\, \dd y
  \end{equation*}
  to get
  \begin{equation*}
    \| \omega_j|_{y=0} \|_{L^2_x}^2
    \le 2
    \left(1+ \left \| \frac{\partial_y \rho_j}{\rho_j} \right \| \right)^2
    \| \omega_j \|_j^2 + 2 \| \partial_y \omega_j \|_j^2.
  \end{equation*}

  By choosing $\mathcal{C}$ large enough and using that
  $\alpha \ge \frac 12$, the result follows after integration over time.
\end{proof}

%

Combining the results, we can conclude this section.
\begin{proof}[Proof of \cref{thm:final-linear-control}]
  Adding the control of \cref{thm:control-u-by-h-linf} with a factor
  \(\epsilon_3 \jp^{-\gamma} / \beta\) and
  \cref{thm:vorticity-control} with a factor
  \(\jp^{1-2\gamma} / \beta^2\) to the inequality of
  \cref{thm:control-u-by-h-l2} yields
  \begin{equation*}
    \begin{aligned}
      &\left(\frac 12 - 4 \epsilon_3
        - \frac{4C_m^2 \| (1+y) \partial_{xyy} U^P \|_{L^\infty_{t,x}L^2_y(\rho_0)}^2}{\jp^{2\gamma-2}\beta^3}
      \right)
      \int_0^T \| u_j \|_j^2 \, \dd t\\
      &\quad + \frac{\epsilon_3}{4\beta \jp^\gamma} \| u_j(T) \|_j^2
      + \left(\frac{\epsilon_3}{2\beta \jp^\gamma} -
        \frac{\epsilon_2}{\beta\jp^\gamma}\right)
      \int_0^T \| \partial_y u_j \|_j^2\, \dd t      - \frac{\epsilon_4}{\beta^2 \jp^{2\gamma}} \int_0^T \|
        \partial_y^2 u_j(t) \|_j^2\, \dd t \\
      &\quad- \left( \frac{\epsilon_1}{\jp^{2\gamma}}
        + \frac{\epsilon_3 C_m \| (1{+}y) \partial_y U^P
          \|_{L^\infty_{t,x}L^2_y(\rho_0)}^2}{\beta^2 \jp^{2\gamma}}
        + \frac{4C_m^2 \| (1+y) \partial_{xyy} U^P \|_{L^\infty_{t,x}L^2_y(\rho_0)}^2}{\jp^{2\gamma}\beta^3}
      \right) \int_0^T \| \partial_x u_j \|_{j+1}^2\, \dd t \\
      &\quad \frac{\jp^{1-2\gamma}}{\beta^2} \| (1+y) \omega_j(T) \|_j^2
      + \int_0^T \frac{\jp^{2-2\gamma}}{\beta} \| (1+y) \omega_j \|_j^2\, \dd t
      + \int_0^T \frac{\jp^{1-2\gamma}}{\beta^2} \| (1+y) \partial_y \omega_j \|_j^2\, \dd t\\
      &\le
      \frac{\beta \jp^\gamma}{\epsilon_3} \| H_j(T) \|_j^2
      \\
      &\quad + \left[
        16 \beta^2 \jp^2
        + \frac{\jp^{2\gamma}}{\epsilon_1} C_m^2
        \| (1{+}y) \partial_y U^P \|_{L^\infty_t L^\infty_xL^2_y(\rho_0)}^2
        + \frac{\beta^2\jp^{2\gamma}}{4\epsilon_4}
      \right] \int_0^T \| H_j \|_j^2 \, \dd t \\
      &\quad+
      \left[
        \frac{\beta\jp^{\gamma}}{4\epsilon_2} + 16 C_l^2 A_{j,\alpha}^2
      \right] \int_0^T \| \partial_y H_j \|_j^2 \, \dd t
      + \int_0^T \| H_j \|_j \| F_j \|_j\, \dd t \\
      &\quad +
      \frac{\epsilon_3}{2\beta\jp^\gamma} \| u_{\init,j} \|_j^2
      + \frac{\epsilon_3}{\beta^2 \jp^{2\gamma}} \int_0^T \| F_j
      \|_j^2\, \dd t \\
      &\quad + \frac{\jp^{1-2\gamma}}{\beta^2} \| \omega_{\init,j}
      \|_j^2
      + \frac{4\jp^{1-2\gamma}}{\beta^2} \int_0^T \| (1+y) F_j \|_j^2\, \dd t
      + \frac{4\jp^{1-2\gamma}}{\beta^2} \int_0^T \| F_j|_{y=0} \|_{L^2_x}^2\, \dd t.
    \end{aligned}
  \end{equation*}

  Using that
  $\partial_x u_j(t) = \left( \frac{j+2}{j+1}\right)^r
  \frac{(j+1)^\gamma}{\tau(t)} u_{j+1}(t)$, we can sum over $j$ and
  choose $\epsilon_1,\epsilon_2,\epsilon_3,\epsilon_4$ appropriately
  to arrive for
  $m > \frac 12, \alpha \ge \frac 12,\gamma \ge 1, \tau_1 > 0, r \in
  \R$ at the control
  \begin{equation}\label{eq:first-summed-u}
    \begin{aligned}
      &\sum_{j=0}^{\infty}
      \left\{
        \int_0^T \| u_j \|_j^2 \, \dd t
        + \frac{1}{\beta \jp^\gamma} \| u_j(T) \|_j^2
        + \frac{\jp^{2-2\gamma}}{\beta} \int_0^T \| (1+y) \omega_j \|_j^2\, \dd t
      \right\}\\
      &\quad+ \sum_{j=0}^{\infty}
      \frac{\jp^{1-2\gamma}}{\beta^2}
      \left\{
        \| (1+y) \omega_j(T) \|_j^2 +
        \int_0^T \| (1+y) \partial_y\omega_j(t) \|_j^2 \, \dd t
      \right\} \\
      &\le C
      \sum_{j=0}^{\infty}
      \left\{
        \beta \jp^\gamma \| H_j(T) \|_j^2
        + \beta^2 \jp^{2\gamma}
        \int_0^T \| H_j \|_j^2 \, \dd t
        + \beta \jp^\gamma \int_0^T \| \partial_y H_j \|_j^2\, \dd t \right\} \\
      &\quad+ C \sum_{j=0}^{\infty}
      \left\{
        \frac{1}{\beta \jp^\gamma}
        \| u_{\init,j} \|_j^2
        + \frac{\jp^{1-2\gamma}}{\beta}
        \| (1+y) \omega_{\init,j} \|_j^2 \right\} \\
      &\quad+ C \sum_{j=0}^{\infty}
      \left\{
        \int_0^T \frac{1}{\beta^2\jp^{2\gamma}} \| F_j \|_j^2\, \dd t
        + \frac{\jp^{1-2\gamma}}{\beta^2} \int_0^T \| (1+y) F_j \|_j^2\, \dd t
        + \frac{\jp^{1-2\gamma}}{\beta^2} \int_0^T \| F_j|_{y=0} \|_{L^2_x}^2\, \dd t.
      \right\}
    \end{aligned}
  \end{equation}
  if
  \begin{equation*}
    \beta \ge \mathcal{C} (1 + \| (U^P,V^P) \|_{low})\,
    (1 + \frac{1}{\tau_1} + \| (U^P,V^P) \|_{low})
    \text{ and }
    \tau(T) \ge \tau_1
  \end{equation*}
  where $C$ and $\mathcal{C}$ are constant only depending on
  $m,\alpha, \gamma,r$ (and not $\tau_1$).

  Controlling $H$ by \cref{thm:control-h-f} then yields the result for
  a fixed time \(T\).  Applying this estimate for all $T$ in $[0,T^*]$
  then shows the claimed estimate.

  For $\gamma \ge 5/4$ we find that $\jp^{2\gamma-4} \ge
  \jp^{1-2\gamma}$ so that
  \begin{equation*}
    \jp^{2\gamma-4} \| F_j \|_j^2
    + \jp^{1-2\gamma} \| (1+y)F_j \|_j^2
    \le 2 \jp^{2\gamma-4} \| (1+\frac{y}{\jp^{2\gamma - \frac{5}{2}}}) F_j \|_j^2,
  \end{equation*}
  which proves the expression in this case.
\end{proof}

  \section{Nonlinear estimates} \label{sec:nonlinear}

  In order to close the estimate, we have to estimate $F_j$.
  \begin{proposition}
    \label{thm:fj-u-gevrey}
    Fix the parameters $m,\alpha,\gamma,r$ and an additional parameter $R$ such that
    \begin{equation} \label{fix_parameters}
      \begin{aligned}
        &\gamma \in [\frac{3}{2},2], \quad    \alpha \le \gamma - 1, \quad   m \ge \frac{2\gamma-1}{\alpha} + 1, \\
        & r > 2\gamma, \quad  R > 2\gamma + 1, \quad     R \ge r + 3\gamma - 2.
      \end{aligned}
    \end{equation}
    Then there exists a constant $C = C(m,\alpha,\gamma,r)$ such that
    for $\beta$, $\tau_1$ and $T$ with $\tau(T) \ge \tau_1$,
    \begin{equation*}
      \begin{aligned}
        &\sum_{j=0}^{\infty}
        \frac{1}{\jp^{4-2\gamma}} \int_0^T
        \left\| \left(1+\frac{y}{\jp^{2\gamma - \frac 52}} \right) F_j \right\|_j^2 \,  \dd t\\
        &\le 2 \int_0^T \| (1+y) f^e_j \|_{\gamma,\tau,r-2+\gamma}^2\, \\
        &+ \frac{C \beta}{\tau_1^4}
        \left[
          \sup_{[0,T]}
          \left(
            \| u \|_{\gamma,\tau,r-\frac \gamma 2}^2
            + \frac{\| (1+y) \omega \|_{\gamma,\tau,r+\frac{1}{2}-\gamma}^2}{\beta}
            + | U^E |_{\gamma,\tau,R}^2
          \right)
        \right]
        \int_0^T
        \left[
          \| u \|_{\gamma,\tau,r}^2
          + \frac{\| (1+y) \omega \|_{\gamma,\tau,r+1-\gamma}^2}{\beta}
        \right] \, \dd t.
      \end{aligned}
    \end{equation*}
  \end{proposition}
  We restrict to the case of $\gamma \ge 3/2$ because we need
  $\alpha \ge \gamma-1$ in order to control the terms
  $\partial_x^k u \partial_x^{l-k+1} u$ in $F_j$. Combined with the
  earlier requirement that $\alpha \ge 1/2$ this yields
  $\gamma \ge 3/2$.
  \begin{proof}
    Write $F_j = f_j^e + \sum_{i=1}^6 F_j^i$ with
    \begin{align*}
      F_j^1 &= M_j \left[u \partial_x, \partial_x^j\right] u
              + M_j \jp\, \partial_x u\, \partial_x^j u, \\
      F_j^2 &= M_j \left[\partial_y u, \partial_x^j\right] v
              + M_j j\, \partial_{xy} u\, \partial_x^{j-1} v
              + M_j v\, \partial_x^j \partial_y u, \\
      F_j^3 &= M_j \left[U^e \partial_x, \partial_x^j\right] u
              + M_j j\, \partial_x U^e\, \partial_x^j u, \\
      F_j^4 &= M_j \left[V^e \partial_y, \partial_x^j\right] u, \\
      F_j^5 &= M_j \left[\partial_x U^e, \partial_x^j\right] u, \\
      F_j^6 &= M_j \left[\partial_y U^e, \partial_x^j \right] v
              + M_j j\, \partial_{xy} U^e\, \partial_x^{j-1} v.
    \end{align*}

    As $\gamma \ge 3/2$ and $\alpha \le \gamma-1$, we have
    $2\gamma - \frac{5}{2} \ge \alpha$, so that
    \begin{equation*}
      \left\| \left(1+\frac{y}{\jp^{2\gamma - \frac 52}} \right) F_j
      \right\|_j
      \le
      \left\| \left(1+\frac{y}{\jp^\alpha}\right) F_j
      \right\|_j
    \end{equation*}
    so that it suffices to bound the right-hand side.

    \noindent\textbf{Analysis of $F_j^1$.}
    We write
    \begin{equation*}
      F^1_j = \sum_{l=2}^{\floor{\frac{j+1}{2}}}
      \binom{j}{l}
      \frac{M_j}{M_l M_{j-l+1}} u_l u_{j-l+1} \: + \:
      \sum_{l=\floor{\frac{j+1}{2}}+1}^{j-1} \binom{j}{l} \frac{M_j}{M_l M_{j-l+1}}
      u_l u_{j-l+1}
      \: =: \: F^1_{j,low}  + F^1_{j,high}.
    \end{equation*}
    For $F_{j,low}^1$, we notice that for $l \le \floor{\frac{j+1}{2}}$
    there exist a constant $C = C(r)$ with
    \begin{equation*}
      \binom{j}{l} \frac{M_j}{M_l M_{j-l+1}}
      \le \frac{C}{\tau_1} \binom{j}{l}^{1-\gamma} \frac{\jp^{\gamma}}{\lp^r}.
    \end{equation*}
    This shows
    \begin{equation*}
      \begin{aligned}
        \frac{1}{\jp^{2-\gamma}}
        \left\|
          \left(1+\frac{y}{\jp^{\alpha}} \right)
          F^1_{j,low} \right\|_j
        &\le \frac{C}{\tau_1} \sum_{l=2}^{\floor{\frac{j+1}{2}}}
        \binom{j}{l}^{1-\gamma} \frac{\jp^{2\gamma-2}}{\lp^r}
        \| u_l u_{j-l+1} \|_{j-1} \\
        &\le \frac{C}{\tau_1} \sum_{l=2}^{\floor{\frac{j+1}{2}}}
        \binom{j}{l}^{1-\gamma} \frac{\jp^{2\gamma-2}}{\lp^r}
        \left\| \left(\frac{\rho_{j-1}}{\rho_{j-l+1}}\right)^{1/2}
          u_l \right\|_{L^\infty_{x,y}}
        \| u_{j-l+1} \|_{j-l+1}.
      \end{aligned}
    \end{equation*}

    Note that for an absolute constant $C_a$,
    \begin{equation} \label{ineq.binom}
      \binom{j}{l}^{1-\gamma} (j+1)^{2\gamma-2}  \le  C_a  \quad \text{for all $2 \le l \le  \floor{\frac{j+1}{2}}$}.
    \end{equation}
    From the 1d Sobolev embedding and \cref{thm:poincare-y-rho}, we find that
    for $n \le \min(m-1,l)$:
    \begin{align*}
      \left\|\left(\frac{\rho_{j-1}}{\rho_{j-l+1}}\right)^{1/2} u_l
      \right\|_{L^\infty_{x,y}}
      &\le C_A \left\| \left(\frac{\rho_{j-1}}{\rho_{j-l+1}}\right)^{1/2} \partial_x u_l \right\|_{L^2_xL^\infty_y} \\
      & \le C_A \, C_{1}\,
        \sup_y \left(\frac{\rho_{j-1} \rho_{n}}{\rho_{l}
        \rho_{j-l+1}}\right)^{1/2}
        \|  \partial_x \partial_y u_l \|_{l} \\
      &\le \frac{C}{\tau_1} \sup_y\left(\frac{\rho_{j-1} \rho_{n}}{\rho_{l}
        \rho_{j-l+1}}\right)^{1/2}
        \lp^\gamma\, \| (1+y) \omega_{l+1} \|_{l+1}
    \end{align*}
    where $C_A$ is an absolute constant, $C$ is a constant
    depending on $m,r$. Note that we used here \cref{thm:rho-j-j-1} to bound $\|\omega_{l+1}\|_l$ by $\|(1+y)\omega_{l+1}\|_{l+1}$.The factor with the $\rho$ is explicit:
    \begin{equation*}
      \left(\frac{\rho_{j-1} \rho_{n}}{\rho_{l}
          \rho_{j-l+1}}\right)^{1/2}
      = \frac{\prod_{k=1}^{l} (1+\frac{y}{k^\alpha})}
      {\prod_{k=j-l+2}^{j-1} (1+\frac{y}{k^\alpha}) \prod_{k=1}^n (1+\frac{y}{k^\alpha})}.
    \end{equation*}

    For $l \le m-1$, we take $n=l$ and find that
    \begin{equation*}
      \left(\frac{\rho_{j-1} \rho_{n}}{\rho_{l}
          \rho_{j-l+1}}\right)^{1/2}
      \le 1.
    \end{equation*}

    For $l > m-1$, we take $n=m-1$ and find that
    \begin{equation*}
      \begin{aligned}
        \left(\frac{\rho_{j-1} \rho_{n}}{\rho_{l}
            \rho_{j-l+1}}\right)^{1/2}
        &\le \frac{\prod_{k=1}^{l} (1+\frac{y}{k^\alpha})}
        {\prod_{k=j-l+2}^{j-m+2} (1+\frac{y}{k^\alpha}) \prod_{k=1}^{m-1}
          (1+\frac{y}{k^\alpha})} \\
        &\le
        \left(
          \frac{(j-l+2)\dotsm(j-m+2)}{m\dotsm l}
        \right)^\alpha \\
        &\le C \binom{j}{l}^\alpha \jp^{-\alpha (m-1)}
      \end{aligned}
    \end{equation*}
    for a constant $C = C(m,\alpha)$ and using that
    $l \le \floor{\frac{j+1}{2}}$.

    Hence we find for a constant $C = C(m,\alpha,r)$
    that
    \begin{equation*}
      \frac{1}{\jp^{2-\gamma}}
      \left\|
        \left(1+\frac{y}{\jp^{\alpha}} \right)
        F^1_{j,low} \right\|_j
      \le \frac{C}{\tau_1^2} \sum_{l=2}^{\floor{\frac{j+1}{2}}}
      \lp^{\gamma-r}
      \| (1+y) \omega_{l+1} \|_{l+1}
      \| u_{j-l+1} \|_{j-l+1}
    \end{equation*}
    using that $1-\gamma+\alpha \le 0$ and $2\gamma-2 \le \alpha(m-1)$.

    The discrete Young's convolution inequality implies for all
    $t \in [0,T]$ that
    \begin{align*}
      &\sum_{j=0}^{\infty} \frac{1}{\jp^{4-2\gamma}}
        \left\|
        \left(1+\frac{y}{\jp^{\alpha}} \right)
        F^1_{j,low} \right\|_j^2\\
      & \le \frac{\mathcal{C}}{\tau_1^4}
        \left( \sum_{l=0}^\infty \lp^{\gamma-r} \| (1+y) \omega_l \|_l
        \right)^2
        \sum_{j=0}^{\infty}  \| u_j \|_j^2 \\
      & \le  \frac{\mathcal{C}}{\tau_1^4}
        \left(\sum_{l=0}^{\infty}  \lp^{4\gamma-1-2r} \right)
        \left(\sum_{l=0}^{\infty}  \lp^{1-2\gamma} \| (1+y) \omega_l \|_l^2\right)
        \left(\sum_{j=0}^{\infty}  \| u_j \|_j^2 \right).
    \end{align*}
    As $4\gamma - 1 - 2r < -1$, the first integral is
    finite. Hence we arrive at the required estimate
    \begin{equation*}
      \sum_{j=0}^{\infty} \frac{1}{\jp^{4-2\gamma}}
      \int_0^T \| F^1_{j,low} \|_j^2\, \dd t
      \le \frac{C}{\tau_1^2}
      \sup_{t\in[0,T]}
      \| (1+y) \omega \|_{\gamma,\tau,r+1-\gamma}^2
      \int_0^T
      \| u \|_{\gamma,\tau,r}^2\,
      \dd t
    \end{equation*}
    with a constant $C=C(m,\alpha,\gamma,r)$.

    For the treatment of $F^1_{j,high}$ swap the roles of $u_l$ and
    $u_{j-l+1}$ so that
    \begin{equation*}
      F_{j,high}^1
      = \sum_{l=2}^{j- \floor{\frac{j+1}{2}}}
      \binom{j}{l-1} \frac{M_j}{M_l M_{j-l+1}}
      u_l u_{j-l+1}.
    \end{equation*}
    In the given range $l=2,\dots,j-\floor{\frac{j+1}{2}}$ we find
    \begin{equation*}
      \binom{j}{l-1} \le \binom{j}{l}
    \end{equation*}
    so that it can be bounded as $F_{j,low}^1$.

    \noindent\textbf{Analysis of $F_j^2$.} We write
    \begin{equation*}
      \begin{aligned}
        F_j^2
        &= - \sum_{l=2}^{\floor{\frac{j+1}{2}}}
        \binom{j}{l} \frac{M_j}{M_l M_{j-l+1}}
        \partial_y u_l\, \partial_x^{-1} v_{j-l+1}
        - \sum_{l=\floor{\frac{j+1}{2}}+1}^{j-1}
        \binom{j}{l} \frac{M_j}{M_l M_{j-l+1}}
        \partial_y u_l\, \partial_x^{-1} v_{j-l+1} \\
        &=: F_{j,low}^2 + F_{j,high}^2
      \end{aligned}
    \end{equation*}
    and note that it vanishes unless $j \ge 3$.

    By $v_{j-l+1} = -\partial_x \int_0^y u_{j-l+1}\, \dd z$ we find for
    $n \le \min(m-1,j-l+1)$ using the 1d Sobolev inequality and
    \cref{thm:poincare-y-rho} that
    \begin{equation*}
      \begin{aligned}
        \left\|
          \left(1+\frac{y}{\jp^{\alpha}} \right)
          \partial_y u_l \partial_x^{-1} v_{j-l+1} \right\|_j
        &\le
        \left\|
          \partial_y u_l \partial_x^{-1} v_{j-l+1} \right\|_{j-1}\\
        &\le C_{m-n}
        \left\|
          \left(\frac{\rho_{j-1} \rho_{n}}{\rho_{l}
              \rho_{j-l+1}}\right)^{1/2}
          \partial_y u_l
        \right\|_{L^\infty_xL^2_y(\rho_{l})}
        \| u_{j-l+1} \|_{j-l+1} \\
        &\le \frac{\mathcal{C}}{\tau_1}
        \sup_y\left(\frac{\rho_{j-1} \rho_{n}}{\rho_{l}
            \rho_{j-l+1}}\right)^{1/2}
        \lp^\gamma\, \| (1+y) \omega_{l+1} \|_{l+1}
        \| u_{j-l+1} \|_{j-l+1}
      \end{aligned}
    \end{equation*}
    for a constant $\mathcal{C} = \mathcal{C}(m,r)$.

    In the range $l=2,\dots,\floor{\frac{j+1}{2}}$ for $F_{j,low}^2$ we
    find that $j-l+1 \ge \frac{j+1}{2}$ and as we can assume that $j \ge
    3$ we can always ensure that this is at least $2$.

    For $\frac{j+1}{2} \le m-1$, we can take $n=2$ and find a constant
    $C = C(m,r)$ such that
    \begin{equation*}
      \sup_y\left(\frac{\rho_{j-1} \rho_{n}}{\rho_{l}
          \rho_{j-l+1}}\right)^{1/2}
      \le C
    \end{equation*}
    and otherwise we can taken $n=m-1$ and find the same control as for
    $F_{j,low}^1$ as
    \begin{equation*}
      \frac{1}{\jp^{2-\gamma}}
      \left\| \left(1+\frac{y}{\jp^{\alpha}} \right)
        F^2_{j,low} \right\|_j
      \le \frac{C}{\tau_1^2} \sum_{l=2}^{\floor{\frac{j+1}{2}}}
      \lp^{\gamma-r}
      \| (1+y) \omega_{l+1} \|_{l+1}
      \| u_{j-l+1} \|_{j-l+1}
    \end{equation*}
    and we can conclude as for $F_{j,low}^1$.

    For $F_{j,high}^2$ we find
    \begin{equation*}
      F_{j,high}^2
      = - \sum_{l=2}^{j-\floor{\frac{j+1}{2}}}
      \binom{j}{l-1} \frac{M_j}{M_l M_{j-l+1}} \partial_x^{-1} v_l\,
      \partial_y u_{j-l+1}.
    \end{equation*}
    For $n = \min(m-1,l+1)$ we find
    \begin{equation*}
      \begin{aligned}
        \left\|
          \left(1+\frac{y}{\jp^{\alpha}} \right)
          \partial_x^{-1} v_l\,
          \partial_y u_{j-l+1} \right\|_{j}
        &\le
        \left\|
          \left(\frac{\rho_{j-1}}{\rho_{j-l}}\right)^{1/2}
          \partial_x^{-1} v_l
        \right\|_{L^\infty_{x,y}}
        \| (1+y) \omega_{j-l+1} \|_{j-l+1} \\
        &\le \frac{C}{\tau_1}
        \sup_y\left(\frac{\rho_{j-1} \rho_{n}}{\rho_{l+1}
            \rho_{j-l}}\right)^{1/2}
        \lp^{\gamma}
        \| u_{l+1} \|_{l+1}
        \| (1+y) \omega_{j-l+1} \|_{j-l+1}.
      \end{aligned}
    \end{equation*}
    For $l+1<m-1$ we can find a constant $C = C(m)$ such that
    \begin{equation*}
      \binom{j}{l-1} \le \binom{j}{l} \jp^{-1}.
    \end{equation*}
    Using the stronger assumption $2\gamma -1 \le \alpha(m-1)$, we can
    then conclude as in the treatment of $F_{j,low}^1$ that
    \begin{equation*}
      \frac{1}{\jp^{2-\gamma}} \left\|
        \left(1+\frac{y}{\jp^{\alpha}} \right)
        F^2_{j,high} \right\|_j
      \le \frac{C}{\tau_1^2} \sum_{l=2}^{j-\floor{\frac{j+1}{2}}}
      \lp^{\gamma-r}
      \| u_{l+1} \|_{l+1}\,
      \jp^{-1} \| (1+y) \omega_{j-l+1} \|_{j-l+1}.
    \end{equation*}
    Hence we find
    \begin{align*}
      &\sum_{j=0}^{\infty} \frac{1}{\jp^{4-2\gamma}}
        \left\| \left(1+\frac{y}{\jp^{\alpha}} \right)
        F^2_{j,high} \right\|_j^2 \\
      & \le \frac{C}{\tau_1^4}
        \left( \sum_{l=0}^\infty \lp^{\gamma-r} \| u_l \|_l
        \right)^2
        \sum_{j=0}^\infty  \jp^{-2} \| (1+y) \omega_j \|_j^2 \\
      & \le  \frac{C}{\tau_1^4}
        \left(\sum_{l=0}^{\infty}  \lp^{3\gamma-2r}\right)
        \left(\sum_{l=0}^{\infty}  \lp^{-\gamma} \| u_l \|_l^2\right)
        \left(\sum_{j=0}^{\infty}  \jp^{-2} \| (1+y) \omega_j \|_j^2 \right).
    \end{align*}
    As $3\gamma-2r < -1$, this gives the required estimate
    \begin{equation*}
      \sum_{j=0}^{\infty} \frac{1}{\jp^{4-2\gamma}}
      \int_0^T \| F^2_{j,high} \|_j^2\, \dd t
      \le \frac{C}{\tau_1^4}
      \sup_{t\in[0,T]}
      \| u \|_{\gamma,\tau,r-\frac{\gamma}{2}}^2
      \int_0^T
      \| (1+y) \omega \|_{\gamma,\tau,r-1}^2\,
      \dd t
    \end{equation*}
    with a constant $C=C(m,\alpha,\gamma,r)$. As
    $r-1 \le r+1-\gamma$ this is the required control.

    \noindent\textbf{Analysis of $F_{j}^3$ and $F_{j}^5$.}
    We write
    \begin{equation*}
      \begin{aligned}
        F_j^3 + F_j^5
        &= - \sum_{l=2}^{j}
        \binom{j}{l} \frac{M_j}{M_l M_{j-l+1}}
        U_l^e u_{j-l+1}
        - \sum_{l=1}^{j}
        \binom{j}{l} \frac{M_j}{M_{l+1} M_{j-l}}
        U_{l+1}^e u_{j-l} \\
        &= - \sum_{l=2}^{\floor{\frac{j+1}{2}}}
        \left[
          \binom{j}{l} + \binom{j}{l-1}
        \right]
        \frac{M_j}{M_l M_{j-l+1}}
        U_l^e u_{j-l+1}
        + \sum_{l=\floor{\frac{j+1}{2}}+1}^{j+1}
        \left[
          \binom{j}{l} + \binom{j}{l-1}
        \right]
        \frac{M_j}{M_l M_{j-l+1}}
        U_l^e u_{j-l+1}\\
        &=: F_{j,low}^{3,5} + F_{j,high}^{3,5}
      \end{aligned}
    \end{equation*}
    with the convention that
    \begin{equation*}
      \binom{j}{j+1} = 0.
    \end{equation*}
    Using the definition of $U^e$ and the 1d Sobolev embedding theorem
    we find
    \begin{equation*}
      \| U_l^e \|_{L^\infty_{x,y}}
      \le \| U^E_l \|_{L^\infty_x}
      \le \frac{C_s \lp^\gamma}{\tau_1}
      \| U^E_{l+1} \|.
    \end{equation*}
    As $l\ge 2$, this implies
    \begin{equation*}
      \left\|
        \left(1+\frac{y}{\jp^{\alpha}} \right)
        U_l^e u_{j-l+1} \right\|_j
      \le \frac{C_s\lp^\gamma}{\tau_1}
      \| U^E_{l+1} \|\, \| u_{j-l+1} \|_{j-l+1}.
    \end{equation*}

    For $l=2,\dots,\floor{\frac{j+1}{2}}$ we find for a constant $C = C(\gamma,r)$
    \begin{equation*}
      \left[
        \binom{j}{l} + \binom{j}{l-1}
      \right]
      \frac{M_j}{M_l M_{j-l+1}}
      \le \frac{C}{\tau_1} \binom{j}{l}^{1-\gamma} \frac{\jp^\gamma}{\lp^r}
    \end{equation*}
    so that as $l \ge 2$
    \begin{equation*}
      \frac{1}{\jp^{2-\gamma}}
      \left\|
        \left(1+\frac{y}{\jp^{\alpha}} \right)
        F_{j,low}^{3,5}
      \right\|_j
      \le \frac{C}{\tau_1^2}
      \sum_{l=2}^{\floor{\frac{j+1}{2}}}
      \lp^{\gamma-r}
      \| U^E_{l+1} \|\, \| u_{j-l+1} \|_{j-l+1}.
    \end{equation*}
    Hence we find
    \begin{align*}
      \sum_{j=0}^{\infty} \frac{1}{\jp^{4-2\gamma}}
      \left\|
      \left(1+\frac{y}{\jp^{\alpha}} \right)
      F_{j,low}^{3,5}
      \right\|_j^2
      & \le \frac{C}{\tau_1^4}
        \left( \sum_{l=0}^\infty \lp^{\gamma-r} \| U^E_l \|
        \right)^2
        \sum_{j=0}^\infty \| u_j \|_j^2 \\
      & \le  \frac{C}{\tau_1^4}
        \left(\sum_{l=0}^{\infty}  \lp^{2\gamma-2R} \right)
        \left(\sum_{l=0}^{\infty}  \lp^{2R-2r} \| U^E_l \|^2\right)
        \left(\sum_{j=0}^{\infty}  \| u_j \|_j^2 \right).
    \end{align*}
    As $2\gamma-R < -1$ this gives the bound
    \begin{equation*}
      \int_0^T \sum_{j=0}^{\infty} \frac{1}{\jp^{4-2\gamma}}
      \left\|
        \left(1+\frac{y}{\jp^{\alpha}} \right)
        F_{j,low}^{3,5}
      \right\|_j^2
      \, \dd t
      \le \frac{C}{\tau_1^4}
      \sup_{t\in[0,T]} |U^E|_{\gamma,\tau,R}^2
      \int_0^T \| u \|_{\gamma,\tau,r}^2\, \dd t.
    \end{equation*}

    For $l=\floor{\frac{j+1}{2}}+1,\dots,j$ we find
    \begin{equation*}
      \left[
        \binom{j}{l} + \binom{j}{l-1}
      \right]
      \frac{M_j}{M_l M_{j-l+1}}
      \le \frac{C}{\tau_1} \binom{j}{l-1}^{1-\gamma} \frac{\lp^\gamma}{(j{-}l{+}1)^r}
    \end{equation*}
    so that
    \begin{equation*}
      \frac{1}{\jp^{2-\gamma}}
      \left\|
        \left(1+\frac{y}{\jp^{\alpha}} \right)
        F_{j,high}^{3,5}
      \right\|_j
      \le \frac{C}{\tau_1^2}
      \sum_{l=\floor{\frac{j+1}{2}}+1}^j
      \lp^{3\gamma-2}
      \| U^E_{l+1} \|\,
      (j{-}l{+}1)^{-r}
      \| u_{j-l+1} \|_{j-l+1}.
    \end{equation*}
    Hence we find
    \begin{align*}
      \sum_{j=0}^{\infty} \frac{1}{\jp^{4-2\gamma}}
      \left\|
      \left(1+\frac{y}{\jp^{\alpha}} \right)
      F_{j,high}^{3,5}
      \right\|_j^2
      & \le \frac{\mathcal{C}}{\tau_1^4}
        \left( \sum_{j=0}^\infty \jp^{-r} \| u_l \|_l
        \right)^2
        \sum_{l=0}^\infty  \lp^{6\gamma-4} \| U^E_{l+1} \| \\
      & \le  \frac{\mathcal{C}}{\tau_1^4}
        \left(\sum_{j=0}^{\infty}  \jp^{-2r} \right)
        \left(\sum_{j=0}^{\infty}  \| u_l \|_l^2\right)
        \left(\sum_{l=0}^{\infty}  \lp^{6\gamma-4} \| U^E_{l+1} \|^2 \right).
    \end{align*}
    As $r > \frac 12$ this gives the bound
    \begin{equation*}
      \int_0^T \sum_{j=0}^{\infty} \frac{1}{\jp^{4-2\gamma}}
      \left\|
        \left(1+\frac{y}{\jp^{\alpha}} \right)
        F_{j,high}^{3,5}
      \right\|_j^2
      \, \dd t
      \le \frac{\mathcal{C}}{\tau_1^4}
      \sup_{t\in[0,T]} |U^E|_{\gamma,\tau,r+3\gamma-2}^2
      \int_0^T \| u \|_{\gamma,\tau,r}^2\, \dd t,
    \end{equation*}
    which is the required bound as $R \ge r+3\gamma-2$.

    \noindent \textbf{Analysis of $F_j^4$.}
    This term is creating trouble with the integrability in $y$ as
    $V^e \sim y$ and is the reason for most technical difficulties.

    We write
    \begin{equation*}
      \begin{aligned}
        F_{j}^4 &= - \sum_{l=1}^{\floor{\frac{j+1}{2}}} \binom{j}{l}
        \frac{M_j}{M_{l+1}M_{j-l}}
        \partial_x^{-1} V^e_{l+1}
        \partial_y u_{j-l}
        \:-\: \sum_{l=\floor{\frac{j+1}{2}}+1}^j \binom{j}{l}
        \frac{M_j}{M_{l+1}M_{j-l}}
        \partial_x^{-1} V^e_{l+1}
        \partial_y u_{j-l} \\
        &=: F_{j,low}^4 + F_{j,high}^4.
      \end{aligned}
    \end{equation*}
    As $l \ge 1$ we find
    \begin{equation*}
      \begin{aligned}
        \left\|
          \left(1+\frac{y}{\jp^{\alpha}} \right)
          \partial_x^{-1} V^e_{l+1}
          \partial_y u_{j-l}
        \right\|_{j}
        &\le
        \left\| \frac{\partial_x^{-1} V^e_{l+1}}{1+y}
        \right\|_{L^\infty_{x,y}}
        \left\| \left(1+\frac{y}{\jp^{\alpha}} \right)
          (1+y) \omega_{j-l} \right\|_j \\
        &\le \frac{C}{\tau} \lp^{\gamma} \| U^E_{l+2} \|\,
        \| (1+y) \omega_{j-l} \|_{j-l}
      \end{aligned}
    \end{equation*}
    where $C=C(r)$ is constant. In the last line we used the 1d Sobolev
    inequality and that
    \begin{equation*}
      \sqrt{\frac{\rho_j}{\rho_{j-l}}}
      \left(1+\frac{y}{\jp^{\alpha}} \right)
      \le C.
    \end{equation*}

    For $l=1,\dots,\floor{\frac{j+1}{2}}$ we find
    \begin{equation*}
      \binom{j}{l} \frac{M_j}{M_{l+1}M_{j-l}}
      \le \frac{C}{\tau_1} \binom{j}{l}^{1-\gamma}
      \lp^{\gamma-r}
    \end{equation*}
    so that
    \begin{equation*}
      \begin{aligned}
        \frac{1}{\jp^{2-\gamma}}
        \left\| \left(1+\frac{y}{\jp^{\alpha}} \right) F_{j,low}^4 \right\|_j
        &\le \frac{C}{\tau_1^2}
        \sum_{l=1}^{\floor{\frac{j+1}{2}}}
        \binom{j}{l}^{1-\gamma}
        \lp^{2\gamma-r}
        \jp^{\gamma-2}
        \| U^E_{l+2} \|\,
        \| (1+y) \omega_{j-l} \|_{j-l} \\
        &\le \frac{C}{\tau_1^2}
        \sum_{l=1}^{\floor{\frac{j+1}{2}}}
        \lp^{2\gamma-r}
        \| U^E_{l+2} \|\,
        \jp^{-1}
        \| (1+y) \omega_{j-l} \|_{j-l}.
      \end{aligned}
    \end{equation*}
    Hence we find
    \begin{equation*}
      \begin{aligned}
        \sum_{j=0}^{\infty} \frac{1}{\jp^{4-2\gamma}}
        \left\| \left(1+\frac{y}{\jp^{\alpha}} \right) F_{j,low}^4 \right\|_j^2
        & \le  \frac{\mathcal{C}}{\tau_1^4}
        \left(\sum_{l=0}^{\infty}  \lp^{2\gamma-r} \| U^E_{l+2}\| \right)^2
        \left(\sum_{j=0}^{\infty}  \jp^{-2} \| (1+y) \omega_j \|_j^2 \right)\\
        & \le  \frac{\mathcal{C}}{\tau_1^4}
        | U^E|_{\gamma,\tau,R}^2
        \left(\sum_{j=0}^{\infty}  \jp^{-2} \| (1+y) \omega_j \|_j^2 \right)\\
      \end{aligned}
    \end{equation*}
    as $4\gamma-2R < -1$. This gives the bound
    \begin{equation*}
      \int_0^T \sum_{j=0}^{\infty} \frac{1}{\jp^{4-2\gamma}}
      \left\| \left(1+\frac{y}{\jp^{\alpha}} \right) F_{j,low}^4 \right\|_j^2
      \, \dd t
      \le \frac{\mathcal{C}}{\tau_1^4}
      \sup_{t\in[0,T]} |U^E|_{\gamma,\tau,R}^2
      \int_0^T \| (1+y) \omega \|_{\gamma,\tau,r-1}^2\, \dd t,
    \end{equation*}
    which is the required bound as $-1 \le 1-\gamma$.

    For $F_{j,high}^4$ we find
    \begin{equation*}
      \frac{1}{\jp^{2-\gamma}}
      \left\| \left(1+\frac{y}{\jp^{\alpha}} \right) F_{j,high}^4 \right\|_j
      \le \frac{C}{\tau^2}
      \sum_{l=\floor{\frac{j+1}{2}}+1}^{j}
      \frac{\lp^{3\gamma-2}}
      {(j{-}l{+}1)^{r}}
      \| U^E_{l+1} \|\,
      \| (1+y) \omega_{j-l} \|_{j-l}.
    \end{equation*}
    As $-1+\gamma-r < - \frac 12$ this gives the bound
    \begin{equation*}
      \int_0^T \sum_{j=0}^{\infty} \frac{1}{\jp^{4-2\gamma}}
      \left\| \left(1+\frac{y}{\jp^{\alpha}} \right) F_{j,high}^4 \right\|_j^2
      \, \dd t
      \le \frac{\mathcal{C}}{\tau_1^4}
      \sup_{t\in[0,T]} |U^E|_{\gamma,\tau,r+3\gamma-2}^2
      \int_0^T \| (1+y)\omega \|_{\gamma,\tau,r+1-\gamma}^2\, \dd t.
    \end{equation*}
    As $R \ge r+3\gamma-2$ this is the required result.

    \noindent\textbf{Analysis of $F_j^6$.} We write
    \begin{equation*}
      F_j^6
      = -\sum_{l=2}^{\floor{\frac{j+1}{2}}} \binom{j}{l}
      \frac{M_j}{M_l M_{j-l+1}}
      \partial_y U^e_l \partial_x^{-1} v_{j-l+1}
      -\sum_{l=\floor{\frac{j+1}{2}}+1}^{j} \binom{j}{l}
      \frac{M_j}{M_l M_{j-l+1}}
      \partial_y U^e_l \partial_x^{-1} v_{j-l+1}
      =: F_{j,low}^6 + F_{j,high}^6.
    \end{equation*}
    As $\partial_y U^e$ is exponentially decaying, we find
    \begin{equation*}
      \begin{aligned}
        \left\|
          \left(1+\frac{y}{\jp^{\alpha}} \right)
          \partial_y U^e_l \partial_x^{-1} v_{j-l+1}
        \right\|_j
        &\le C \| U^E_l \|_{L^\infty_x}
        \| u_{j-l+1} \|_{j-l+1}\\
        &\le \frac{C\lp^{\gamma}}{\tau}
        \| U^E_{l+1} \|\,
        \| u_{j-l+1} \|_{j-l+1}.
      \end{aligned}
    \end{equation*}
    For $F_{j,low}^6$ we find (using that $\binom{j}{l} \frac{M_j}{M_l M_{j-l+1}}\le C (l+1)^{-r}$ for  $l=2...\floor{\frac{j+1}{2}}$):
    \begin{equation*}
      \frac{1}{\jp^{2-\gamma}}
      \left\| \left(1+\frac{y}{\jp^{\alpha}} \right) F_{j,low}^6 \right\|_j
      \le \frac{C}{\tau_1^2}
      \sum_{l=2}^{\floor{\frac{j+1}{2}}+1}
      \lp^{\gamma-r}
      \| U^E_{l+1} \|\,
      \| u_{j-l+1} \|_{j-l+1}.
    \end{equation*}
    As $\gamma-R < - \frac 12$ this gives the control
    \begin{equation*}
      \int_0^T \sum_{j=0}^{\infty} \frac{1}{\jp^{4-2\gamma}}
      \left\| \left(1+\frac{y}{\jp^{\alpha}} \right) F_{j,low}^6 \right\|_j^2
      \, \dd t
      \le \frac{\mathcal{C}}{\tau_1^4}
      \sup_{t\in[0,T]} |U^E|_{\gamma,\tau,R}^2
      \int_0^T \| u \|_{\gamma,\tau,r}^2\, \dd t.
    \end{equation*}

    For $F_{j,high}^6$ we find
    \begin{equation*}
      \frac{1}{\jp^{2-\gamma}}
      \left\| \left(1+\frac{y}{\jp^{\alpha}} \right) F_{j,high}^6 \right\|_j
      \le \frac{C}{\tau_1^2}
      \sum_{l=\floor{\frac{j+1}{2}}+1}^{j}
      \lp^{2\gamma-2}
      \| U^E_{l+1} \|\,
      \jp^{\gamma-r}
      \| u_{j-l+1} \|_{j-l+1}.
    \end{equation*}
    As $\gamma-r < - \frac 12$ this gives the control
    \begin{equation*}
      \int_0^T \sum_{j=0}^{\infty} \frac{1}{\jp^{4-2\gamma}}
      \left\| \left(1+\frac{y}{\jp^{\alpha}} \right) F_{j,high}^6 \right\|_j^2
      \, \dd t
      \le \frac{\mathcal{C}}{\tau_1^4}
      \sup_{t\in[0,T]} |U^E|_{\gamma,\tau,r+1-\gamma}^2
      \int_0^T \| u \|_{\gamma,\tau,r}^2\, \dd t,
    \end{equation*}
    which is the required control as $R \ge r+1-\gamma$.
  \end{proof}
  As a direct consequence of \cref{thm:final-linear-control} and
  \cref{thm:fj-u-gevrey}, we can state the following corollary, where
  we use that \(F_j|_{y=0} = f^e_j|_{y=0}\) as \(u\) and \(v\) vanish
  at \(y=0\).
  \begin{cor} \label{cor:apriori}
    Fix the parameters $m, \alpha, \gamma, r, R$ as in
    \eqref{fix_parameters} and $\alpha \ge 1/2$. There exists
    $\mathcal{C}$ and $\mathbf{C}$ such that for all $\beta, \tau_1, T$
    with
    $$  \beta \ge \mathcal{C} (1 + \| (U^P,V^P) \|_{low})\,
    (1 + \frac{1}{\tau_1} + \| (U^P,V^P) \|_{low}),
    \quad  \text{ and }
    \tau(T) \ge \tau_1
    $$
    we have
    \begin{equation}
      \begin{aligned}
        \tnorm{u}^2 & \le \mathbf{C}  \left[ \frac{1}{\beta} \|
          u_{\init} \|_{\gamma,\tau_0,r+\gamma-\frac{3}{2}}^2 +
          \frac{1}{\beta^2} \| (1+y) \omega_{\init}
          \|_{\gamma,\tau_0,r+\frac 12 - \gamma}^2\right] \\
        &+ \mathbf{C} \left[
          +  \frac{1}{\beta^2} \int_0^T \| f^e_j|_{y=0}
          \|_{\gamma,\tau,r-2+\gamma}^2\, \dd t
          +  \frac{1}{\beta^2} \int_0^T \| (1+y) f^e_j
          \|_{\gamma,\tau,r-2+\gamma}^2 \, \dd t\right] \, \\
        & + \frac{\mathbf{C}}{\tau_1^4}   \left(\frac{1}{\beta}|U^E|_{\gamma,\tau,R}^2 +  \tnorm{u}^2 \right) \tnorm{u}^2
      \end{aligned}
    \end{equation}
    where
    \begin{equation} \label{def:tnorm}
      \begin{aligned}
        \tnorm{u}^2  =
        &\int_0^T \| u \|_{\gamma,\tau,r}^2\, \dd t
        + \sup_{t\in[0,T]}
        \frac{1}{\beta} \| u \|_{\gamma,\tau,r-\frac{\gamma}{2}}^2
        +  \frac{1}{\beta} \int_0^T \ \| (1+y) \omega \|_{\gamma,\tau,r+1-\gamma}^2\, \dd t\\
        &\quad+ \sup_{t\in[0,T]}
        \frac{1}{\beta^2}
        \| (1+y) \omega \|_{\gamma,\tau,r+\frac 12 - \gamma}^2
        + \frac{1}{\beta^2}
        \int_0^T \| (1+y) \partial_y \omega \|_{\gamma,\tau,r+\frac 12 -\gamma}^2\, \dd t
      \end{aligned}
    \end{equation}
  \end{cor}

  \section{Control of the low norm and final a priori estimate}
  \label{sec:low-norm}
  \Cref{cor:apriori}, which shows an a priori bound on the Gevrey norm
  of $u$, was derived under a lower bound on $\beta$ involving
  $\| (U^P,V^P) \|_{low}$. The last step is to see how this low norm
  relates to $\tnorm{u}$. A convenient approach is to establish an
  additional estimate on a weighted Sobolev norm, namely
  \begin{equation*}
    \|f \|^2_{\mathcal{H}^s}
    = \sum_{|\bar{\alpha}|\le s} \int_{\T \times \R_+} |\pa^{\bar{\alpha}} f|^2 (1+y)^{2\bar{\alpha}_2} \rho_0(y) \,\dd x\, \dd y,
  \end{equation*}
  where the summation variable is the multiindex
  $\bar{\alpha} = (\bar{\alpha}_1,\bar{\alpha}_2)$. In this setting, we
  can state the following estimate.
  \begin{lemma} \label{lemma:low}
    Let $s \ge 3$ be an even integer, $m\ge s+2$, $\alpha \ge 0$,
    $r \in \R$, $\gamma \ge 1$, and define $\tnorm{u}$ as in
    \eqref{def:tnorm}. Then, there exists $C$ depending on
    $s,m,\alpha,\gamma,r$ such that
    \begin{equation}
      \begin{aligned}
        \frac{\dd}{\dd t}\| \omega^P \|^2_{\mathcal{H}^s} + \| \pa_y \omega^P \|^2_{\mathcal{H}^s} &  \le C \| \omega^P \|_{\mathcal{H}^s}^s + C (1+ \|U^E \|_{H^{s+1}(\T)} + \tnorm{u}) \| \omega^P \|_{\mathcal{H}^s}^2 \\
        & + \sum_{l=0}^{\frac{s}{2}} \|\pa_t^l (\pa_t + U^E \pa_x U^E) \|^2_{H^{s-2l}}.
      \end{aligned}
    \end{equation}
    where $\omega^P = \pa_y U^P$.
  \end{lemma}
  \begin{proof}
    A similar estimate was established in
    \cite[Proposition~5.6]{masmoudi-wong-2015-local-in-time-prandtl}, so
    that we will only explain the main steps. The starting point is the
    advection-diffusion equation on the vorticity
    \begin{equation}
      (\pa_t  + U^P \pa_x + V^P \pa_y) \omega^P - \pa^2_y \omega^P = 0.
    \end{equation}
    One applies $\pa^\alpha$ to the equation, test it against
    $(1+y)^{2\alpha_2} \rho_0 \pa^\alpha\omega^P$, and sum over
    $|\alpha| \le s$. Then,
    \begin{align*}
      \frac{1}{2} \pa_t \|\omega^P \|_{\mathcal{H}^s}^2 & + \| \pa_y \omega^P \|^2_{\mathcal{H}^s} \le  \sum_{|\alpha| \le s} \int V^P \pa_y ((1+y)^{2\alpha_2} \rho_0) |\pa^\alpha \omega^P|^2   \\
                                                        &  -\sum_{|\alpha| \le s}\int[  \pa^\alpha ,  (U^P \pa_x + V^P \pa_y) ] \omega^P  \,   \pa^\alpha \omega^P (1+y)^{2\alpha_2} \rho_0  \\
                                                        &   -  \sum_{|\alpha| \le s} \int_{\Omega} \pa_y ((1+y)^{2\alpha_2} \rho_0) \pa_y \pa^\alpha \omega^P  \,   \pa^\alpha \omega^P  -  \sum_{|\alpha| \le s} \int_{\{y=0\}}  \pa_y \pa^\alpha \omega^P  \, \pa^\alpha \omega^P.
    \end{align*}
    Using the equation on $U^P$, one can obtain recursively boundary
    conditions for the odd derivatives $\pa_y^{2k+1} \omega^P$, starting
    from the Neumann condition
    $$  \pa_y \omega^P\vert_{y=0} = - \pa_t U^E - U^E \pa_x U^E.$$
    More precisely, the boundary data $\pa_y^{2k+1} \omega^P\vert_{y=0}$
    can be expressed in terms of the data $U^E$ and of products of mixed
    derivatives $\pa_x^{\rho_1} \pa_y^{\rho_2}\omega^P\vert_{y=0}$ with
    $\rho_2\le 2k-2$.  We refer to \cite[Lemma
    5.9]{masmoudi-wong-2015-local-in-time-prandtl} for the expressions
    of these boundary conditions. This allows to establish the following
    bound, {\it cf} equations (5.20)-(5.22) in
    \cite{masmoudi-wong-2015-local-in-time-prandtl}:
    $$ -\sum_{|\alpha| \le s}  \int_{\{y=0\}}  \pa_y \pa^\alpha \omega^P  \, \pa^\alpha \omega^P \le C_s \|\omega^P \|_{\mathcal{H}^s}^s + C_s \sum_{l=0}^{\frac{s}{2}} \|\pa_t^l (\pa_t + U^E \pa_x U^E) \|_{H^{s-2l}}^2  + \frac{1}{4} \|\pa_y \omega^P \|_{\mathcal{H}^s}^2. $$
    The diffusion term does not raise any difficulty: we find
    $$  -  \sum_{|\alpha| \le s} \int_{\Omega} \pa_y ((1+y)^{2\alpha_2} \rho_0) \pa_y \pa^\alpha \omega^P  \,   \pa^\alpha \omega^P  \le  C \|\omega^P\|_{\mathcal{H}^s}^2 + \frac{1}{4}  \|\pa_y\omega^P \|_{\mathcal{H}^s}^2. $$
    where $C$ depends on $s$ and $m$. Also, through standard estimates, we find
    $$ \sum_{|\alpha| \le s} \int V^P \pa_y ((1+y)^{2\alpha_2} \rho_0) |\pa^\alpha \omega^P|^2  \le C \|\omega^P\|_{\mathcal{H}^s}^3 $$
    and
    $$ \sum_{|\alpha| \le s}\int[\pa^\alpha ,  U^P \pa_x  ] \omega^P  \,   \pa^\alpha \omega^P (1+y)^{2\alpha_2} \rho_0 \le C \|\omega^P\|_{\mathcal{H}^s}^3. $$
    The other part of the commutator is slightly more delicate. First, one can show that
    $$  \sum_{|\alpha| \le s, \alpha_1 \neq s}\int[\pa^\alpha ,  V^P \pa_y  ] \omega^P  \,   \pa^\alpha \omega^P (1+y)^{2\alpha_2} \rho_0 \le C \|\omega^P\|_{\mathcal{H}^s}^3. $$
    Note that the weight (that grows with the number of $y$-derivatives)
    allows to compensate for the linear growth in $y$ of $V^P$.  The
    success of this trick comes from the fact that we are interested
    here in Sobolev estimates (contrary to the former Gevrey estimates).
    When $\alpha_1 = s$, namely $\alpha = (s,0)$, one can show similarly
    that
    \begin{align*}
      & \int[\pa^\alpha ,  V^P \pa_y  ] \omega^P  \,   \pa^\alpha \omega^P  \rho_0 - \int \pa^\alpha V^P \pa_y \omega_P    \,   \pa^\alpha \omega^P  \rho_0
        \: \le \:  C  \|\omega^P\|_{\mathcal{H}^s}^3.
    \end{align*}
    However, the term where the $s$ derivatives with respect to $x$
    apply to $V^P$ can not be handled with usual manipulations. It is
    the well-known loss of $x$-derivative peculiar to the Prandtl
    equation: in particular, one cannot control
    $\| (1+y)^{-1}\pa^s_x V^P \|_{L^2(\rho_0)}$ by
    $\|\omega^P\|_{\mathcal{H}^s}$. This is where $\tnorm{u}$ is
    involved. We find that
    \begin{align*}
      \int \pa^\alpha V^P \pa_y \omega_P    \,   \pa^\alpha \omega^P  \rho_0
      &\le \|(1+y)^{-1} \pa_x^s V^P\|_{\infty}\, \| (1+y) \pa_y \omega^P\|_{L^2_xL^2(\rho_0)}\, \|\omega^P\|_{L^2_xL^2(\rho_0)} \\
      &\le C (\|U^E\|_{H^{s+1}} +\tnorm{u})   \|\omega^P\|_{\mathcal{H}^s}^2
    \end{align*}
    using that
    $\|(1+y)^{-1} \pa_x^s V^P\|_{\infty} \le C \|\pa_x^{s+1}u^P\|_\infty
    \le C \tnorm{u}$ as soon as $m \ge s+2$. Putting together the
    previous estimates yields the result.
  \end{proof}
  We conclude this section with
  \begin{proposition} \label{thm:complete-apriori}
    Let us fix $s=6$, $m \ge s+2$, $\alpha$, $\gamma$, $r$, $R$ as in
    \eqref{fix_parameters} and $\alpha \ge 1/2$. Further fix
    $\tau_1 >0$. Let
    \begin{equation*}
      M_\init = 2 \max(\mathbf{C},1) \left( \| u_{\init} \|_{\gamma,\tau_0,r+\gamma-\frac{3}{2}}^2 +  \| (1+y) \omega_{\init} \|_{\gamma,\tau_0,r+\frac 12 - \gamma}^2 + \|\omega^P\vert_{t=0}\|_{\mathcal{H}^s}^2\right)
    \end{equation*}
    where $\mathbf{C}$ is the constant appearing in
    \cref{cor:apriori}. There exists $\beta_*$ and $T_*$ depending on
    $\tau_1$, $M_\init$, on $\|\omega^P\vert_{t=0}\|_{\mathcal{H}^s}$, on $\sup_{[0,T_0]} |U^E|_{\gamma,\tau_0,R}^2$ and on various Sobolev norms of $U^E$, such that, for all $\beta > \beta_*$
    and for all $T \le T_*$ with $\tau(T) \ge \tau_1$: if
    $\tnorm{u}^2 \le \frac{2M_\init}{\beta}$, then
    $\tnorm{u}^2 \le \frac{3M_\init}{2\beta}$.
  \end{proposition}
  \begin{proof}
    Let $\beta, T$ such that
    $\tnorm{u}^2 \le \frac{2M_\init}{\beta} \le 2 M_\init$ (assuming
    $\beta \ge 1$).  We first apply \cref{lemma:low}, which yields
    \begin{equation}
      \begin{aligned}
        \frac{\dd}{\dd t}\| \omega^P \|^2_{\mathcal{H}^s} + \| \pa_y \omega^P \|^2_{\mathcal{H}^s} &  \le C \| \omega^P \|_{\mathcal{H}^s}^s + C (1+ \|U^E \|_{H^{s+1}(\T)} + \sqrt{2M_\init}) \| \omega^P \|_{\mathcal{H}^s}^2 \\
        & + \sum_{l=0}^{\frac{s}{2}} \|\pa_t^l (\pa_t + U^E \pa_x U^E) \|^2_{H^{s-2l}}.
      \end{aligned}
    \end{equation}
    Integrating this differential inequality shows
    \begin{equation} \label{eq:omegaP}
      \sup_{t \in [0,T]} \|\omega^P(t) \|_{\mathcal{H}^s} \le 2 \|\omega^P\vert_{t=0} \|_{\mathcal{H}^s}
    \end{equation}
    for $T \le T_1$, where $T_1$ depends on $M_\init$,
    $\sup_{t \in [0,T_0]} \|U^e(t) \|_{H^{s+1}(\T)}$,
    $\int_0^{T_0} \|\pa_t^l (\pa_t + U^E \pa_x U^E) \|^2_{H^{s-2l}}\, \dd t$ and
    on $\|\omega^P\vert_{t=0} \|_{\mathcal{H}^s}$.

    Standard Sobolev imbeddings imply that
    $$    \max_{0\le k \le 3}\| \pa_x^k U^P \|_{\infty} + \max_{0\le k \le 2}\left\| \frac{\pa_x^k V^P}{1+y} \right\|_{\infty}  \le C \left(   \max_{0\le k \le 3} \|\pa_x^k U^E\|_{\infty} + \tnorm{u}\right). $$
    As regards the other terms defining $\|(U^P,V^P)\|_{low}$, {\it cf} \eqref{def:low}, they all involve $\omega^P$ and are controlled by $\|\omega^P\|_{\mathcal{H}^s}$ as soon as  $s \ge 5$. Hence, it follows from \eqref{eq:omegaP} that
    $$ \|(U^P,V^P)\|_{low} \le K $$
    for $T \le T_1$ and for some $K$ depending on $M_\init$, $\|\omega^P\vert_{t=0} \|_{\mathcal{H}^s}$ and various norms of $U^E$. If we now choose
    $$ \beta_* \ge  \mathcal{C} (1+K)\,   (1 + \frac{1}{\tau_1} +K),
    \quad  \text{ and }   \tau(T) \ge \tau_1
    $$
    where $\mathcal{C}$ is the constant appearing in   \cref{cor:apriori}, we obtain for $\beta \ge \beta_*$:
    $$ \tnorm{u}^2 \le \frac{M_\init}{2\beta}  +   \frac{\mathbf{C}}{\beta^2} \int_0^T \| (1+y) f^e_j \|_{\gamma,\tau,r-2+\gamma}^2 \,\dd t
    + \frac{\mathbf{C}}{\beta \tau_1^4}   \left(|U^E|_{\gamma,\tau,R}^2 + 2M_\init \right) \tnorm{u}^2. $$
    Taking  $\beta_*$ large enough so that
    $$  \frac{\mathbf{C}}{\beta_* \tau_1^4}  \sup_{t \in [0,T_0]} \left(|U^E|_{\gamma,\tau_0,R}^2 + 2M_\init \right) \le \frac{1}{2}$$
    we get
    $$  \tnorm{u}^2 \le M_\init +  \frac{2 \mathbf{C}}{\beta^2} \int_0^T
    \| (1+y) f^e_j \|_{\gamma,\tau,r-2+\gamma}^2 \, \dd t
    +  \frac{2 \mathbf{C}}{\beta^2} \int_0^T
    \| f^e_j|_{y=0} \|_{\gamma,\tau,r-2+\gamma}^2 \, \dd t
    $$
    If we take $T_* \le T_1$ such that $2 \mathbf{C} \int_0^{T_*} \| (1+y)
    f^e_j \|_{\gamma,\tau,r-2+\gamma}^2\, \dd t \le \frac{1}{2} M_\init$,
    the result follows.
  \end{proof}

  \section{Existence and uniqueness} \label{sec:exist_unique}
  On the basis of the previous  {\it a priori} estimates, we now complete the proof of \cref{thm:main}: we construct a unique solution of \eqref{eq:prandtl}-\eqref{eq:BC} with data $U^P_\init$. This obviously amounts to constructing a unique solution of \eqref{eq:u}-\eqref{eq:BC:bis} with data $u_\init := U^P_\init - U^e\vert_{t=0}$.

  We fix $s=6$, $\gamma=2$.  We take $m \ge s+2$ and $\displaystyle \alpha \ge \frac{1}{2}$ that satisfy  the inequalities  in the first line of \eqref{fix_parameters}. Let $0 < \tau_1 < \tau_0$, $r \in \R$, $T_0 > 0$, and $\displaystyle U^E$, $U^P_\init = u_\init + U^e\vert_{t=0}$ satisfying the assumptions of the theorem. Let now $(\tau'_0, \tau'_1)$ with
  $\displaystyle 0 < \tau_1 < \tau'_1 < \tau'_0 < \tau_0$. Let $r'$ and $R'$ as in  the second line of \eqref{fix_parameters}. As $\tau_0 > \tau'_0$, we have
  \begin{equation*}
    \begin{aligned}
      & \| u_{\init} \|_{\gamma,\tau'_0,r'+\gamma-\frac{3}{2}}^2 +  \| (1+y) \omega_{\init} \|_{\gamma,\tau'_0,r+\frac 12 - \gamma}^2
      \le  C \left( \| u_{\init} \|_{\gamma,\tau_0,r}^2 +  \| (1+y) \omega_{\init} \|_{\gamma,\tau_0,r}^2 \right) < +\infty
    \end{aligned}
  \end{equation*}
  while
  $$  \|\omega^P\vert_{t=0}\|_{\mathcal{H}^s}^2 \le C (  \sup_{[0,T_0]}
  |U^E|_{2,\tau_0,r} + \| (1+y)^{m+6} \omega_\init \|_{H^6(\T \times \R_+)}) < +\infty $$
  and
  \begin{equation*}
    \sup_{[0,T_0]} |U^E|_{2,\tau'_0,R'} \le C  \sup_{[0,T_0]} |U^E|_{2,\tau_0,r} < +\infty
  \end{equation*}
  for a constant $C$ possibly depending on $\tau_0, \tau_0', r, r', R'$.

  The idea is then to apply \Cref{thm:complete-apriori} to a solution of
  an approximate system, for which well-posedness is granted.  Inspired
  by \cite{masmoudi-wong-2015-local-in-time-prandtl}, we consider the
  regularized equation
  \begin{equation} \label{eq:u-reg}
    \partial_t u + (u \partial_x + v \partial_y) u   + (U^e_\epsilon \partial_x + V^e_\epsilon \partial_y) u
    + (u \partial_x  + v \partial_y) U^e_\eps - \epsilon \partial_x^2 u - \partial^2_y u = f^e_\eps,
  \end{equation}
  adding a tangential diffusion $- \epsilon \partial_x^2 u$. The modified vector field $(U^e_\epsilon, V^e_\epsilon)$ takes the form
  $$ U^e_\epsilon =  \pa_y (e^{-\eps y} (y + e^{-y} -1))  U^E_\eps, \quad V^e_\epsilon = - e^{-\eps y} (y + e^{-y} -1) \pa_x U^E_\eps $$
  where  $U^E_\eps$ is an analytic approximation of $U^E$, converging to $U^E$ in the norm $| \ |_{2,\tau_0, r}$ as $\eps \rightarrow 0$.
  Note that $(U^e_\eps, V^e_\eps)$ is still divergence-free, but has now fast decay  in $y$, so that all difficulties generated by the linear growth of $V^e$ vanish.  Accordingly,  the right-hand side  $f^e$ is modified into $f^e_\eps$ replacing $U^E$ by $U^E_\eps$, resp.  $(U^e, V^e)$ by $(U^e_\eps, V^e_\eps)$ in \eqref{def:fe}. Similarly, one regularizes the initial data to obtain some $u_{\init,\eps}$ real analytic in $x,y$, with fast decay at  infinity in $y$ (and obeying suitable compatibility conditions).

  One can show that system \eqref{eq:u-reg} is well-posed following classical methods for fully parabolic equations. For instance, for $T_{\eps,max}$ small enough,  one can prove the existence of a Sobolev solution $u_\eps$ on $(0, T_{\eps,max})$ through a fixed point argument applied to
  $$ \mathcal{T}_\eps u(t) = e^{t (\eps \pa_x^2 + \pa^2_y)} u_{\init,\eps} + \int_0^t e^{(t-s) (\eps \pa_x^2 + \pa^2_y)} F_\eps[u](s) ds $$
  with $F_\eps[u] = f^e_\eps  - (u \partial_x + v \partial_y) u   - (U^e_\epsilon \partial_x + V^e_\epsilon \partial_y) u
  - (u \partial_x  + v \partial_y) U^e_\eps$.
  Moreover, $u^\eps$ remains (real) analytic in $(x,y)$ {\em as long as the Sobolev norm of $u_\eps$ does not blow up}, that is on $(0,T_{\eps,max})$. This property, related to the analytic regularization of the heat kernel is well-known, even in the more difficult context of the Navier-Stokes equation: see \cite{FoiTem,Lombardo,KuVi2} and references therein.

  We now claim that all {\it a priori estimates} obtained for a solution $u$ of \eqref{eq:u} can be established for $u_\eps$ solution of \eqref{eq:u-reg}, uniformly in $\eps$. For this, one just needs to adapt the definitions of the auxiliary quantities $H_j$ and $\phi_j$:  we rather consider
  \begin{equation}
    \label{eq:def-h-j-reg}
    \begin{lgathered}
      \Big(\partial_t + \beta \jp + U^P \partial_x + \jp \partial_x U^p
      + V^P \partial_y - \epsilon \partial_x^2 - \partial_y^2\Big)
      \int_0^y H_j \, \dd z
      = \int_0^y u_j \, \dd z, \\
      H_j|_{t=0} = 0, \qquad
      \partial_y H_j|_{y=0} = 0, \qquad
      H_j|_{y\to \infty} = 0.
    \end{lgathered}
  \end{equation}
  and
  \begin{equation}
    \label{eq:def-phi-j-reg}
    \begin{lgathered}
      \left(-\partial_t + \beta \jp - U^P \partial_x + j \partial_x U^p
        - V^P \partial_y - \partial_y V^P - V^P \frac{\partial_y
          \rho_j}{\rho_j}
        - \epsilon \partial_x^2
        - \left(\partial_y + \frac{\partial_y \rho_j}{\rho_j}\right)^2 \right)
      \phi_j = H_j, \\
      \phi_j|_{t=T} = 0, \qquad
      \phi_j|_{y=0} = 0, \qquad
      \phi_j|_{y\to \infty} = 0.
    \end{lgathered}
  \end{equation}
  The additional good terms coming from $-\eps\pa^2_x$ allow to control the extra commutator terms that it generates. Hence, we can apply \cref{thm:complete-apriori} with $\tau'_0, \tau'_1$, $r'$ and $R'$ instead of $\tau_0, \tau_1, r$, and $R$. Let $\beta_*$ and $T_*$ given by the proposition (note that they are independent of $\eps$). We then introduce
  $$ T_{\eps, *} = \sup \{ T \le T_{\eps, max}, \tnorm{u}^2  \le 2 M_\init/\beta \} $$
  where $\beta > \beta_*$ is fixed, and $\tnorm{u}$  is defined in \eqref{def:tnorm}. Note that $\tnorm{u}$ implicitly depends on $T$.  By continuity in time of $u^\eps$, one has $T_{\eps, *} > 0$. But from \cref{thm:complete-apriori}, one deduces easily that for any $T \le T_*$,
  $\: T_{\eps, max} \ge T_{\eps, *} \ge T$.

  From there, by standard compactness arguments, one obtains a solution to the Prandtl system over $[0,T]$, with the regularity properties stated in the theorem.
  It remains to show uniqueness. For this, we take two solutions $u^1$ and $u^2$ up to time $T$. The difference $u^d$ then satisfies (from
  \eqref{eq:u})
  \begin{equation*}
    \partial_t u^d + u^1 \partial_x u^d + u^d \partial_x u^2
    + v^d \partial_y u^1 + v^2 \partial_y u^d
    + (U^e \partial_x + V^e \partial_y) u^d
    + (u^d \partial_x  + v^d \partial_y) U^e - \partial^2_y u^d = f^{d,e}.
  \end{equation*}
  We then find for $u^d_j$ that
  \begin{equation*}
    \Big(\partial_t + \beta \jp + U^{1,P} \partial_x + \jp \partial_x U^{1,P}
    + V^{2,P} \partial_y - \partial_y^2\Big) u^d_j
    + \partial_y U^{1,P} v^d_j + j \partial_{xy} U^{1,P} \partial_x^{-1} v^d_j
    = F^d_j + \partial_x u^d u_j^d,
  \end{equation*}
  where again $F^d_j$ consists of $f_j^{d,e}$ and mixed terms with less
  than $j$ derivatives on $u^1$, $u^2$ or $u^d$. Comparing with \eqref{eq:u-j}, we see that the only difference is the replacement of $(U^P, V^P)$ by $(U^{1,P}, V^{2,P}$). Let us stress that the  latter field not being divergence-free is not an issue: none of the {\it a priori} estimates carried in \cref{sec:linear} and \cref{sec:nonlinear}  were using the fact that $(U^P, V^P)$ was divergence-free.  One can therefore obtain a similar Gevrey bound on $u^d$, under a lower bound on $\beta$ (involving the low norms of $(U^{1,P}, V^{1,P})$ and $(U^{2,P}, V^{2,P})$). This provides a stability estimate which shows uniqueness.
  These considerations now finish the proof of our main result
  \cref{thm:main}.

  \section*{Acknowledgements}
The authors thank Weiren Zhao for useful remarks on the first version of this manuscript. 
  They acknowledge the support of the Universit\'e Sorbonne Paris
  Cit\'e, through the funding ``Investissements d'Avenir'', convention
  ANR-11-IDEX-0005. H.D. is grateful to the People Programme (Marie
  Curie Actions) of the European Union’s Seventh Framework Programme
  (FP7/2007-2013) under REA grant agreement n. PCOFUND-GA-2013-609102,
  through the PRESTIGE programme coordinated by Campus France.
  D.G.-V. acknowledges the support of the Institut Universitaire de
  France.

  \bibliographystyle{abbrv}
  \bibliography{lit}
\end{document}